\documentclass[a4paper,10pt]{amsart}

\usepackage[matrix,arrow,tips,curve]{xy}
\usepackage[english]{babel}
\usepackage[leqno]{amsmath}
\usepackage{amssymb,amsthm}
\usepackage{amscd}
\usepackage{enumerate}

\input{xy}
\xyoption{all}

\newtheorem{thm}{Theorem}[subsection]
\newtheorem{lemma}[thm]{Lemma}
\newtheorem{lemmadefi}[thm]{Lemma - Definition}

\newtheorem{cor}[thm]{Corollary}

\newtheorem{fact}[thm]{Fact}

\newtheorem{teoalpha}{Theorem}

\theoremstyle{remark}
\newtheorem{remark}[thm]{Remark}
\theoremstyle{definition}
\newtheorem{defi}[thm]{Definition}
\newtheorem{nota}[thm]{}

\numberwithin{equation}{section}

\newcommand{\wt}{\widetilde}
\newcommand{\wh}{\widehat}
\newcommand{\ov}{\overline}
\newcommand{\un}{\underline}

\newcommand{\Aut}{\operatorname{Aut}}

\newcommand{\Hom}{\operatorname{Hom}}

\newcommand{\Spec}{\operatorname{Spec}}

\newcommand{\Pic}{\operatorname{Pic}}
\newcommand{\Jac}{\operatorname{Jac}}

\newcommand{\Set}{{\operatorname{Set}}^1}

\def \val{{\rm val}}
\def \cyc{{\rm cyc}}

\newcommand{\calA}{{ \mathcal A}}

\newcommand{\calE}{{ \mathcal E}}

\newcommand{\calG}{{ \mathcal G}}

\newcommand{\calI}{{ \mathcal I}}

\newcommand{\calN}{{ \mathcal N}}
\newcommand{\calO}{{ \mathcal O}}
\newcommand{\calP}{ {\mathcal P} }

\newcommand{\calX}{{\mathcal X}}
\newcommand{\calY}{{\mathcal Y}}

\newcommand{\bbA}{{\mathbb A}}

\newcommand{\bbG}{{\mathbb G}}

\newcommand{\bbN}{{\mathbb N}}

\newcommand{\bbP}{{\mathbb P}}
\newcommand{\bbQ}{{\mathbb Q}}
\newcommand{\bbR}{{\mathbb R}}

\newcommand{\bbZ}{{\mathbb Z}}

\newcommand{\R}{{\mathbb R}}
\newcommand{\Z}{{\mathbb Z}}

\newcommand{\m}{\mathfrak{m}}

\def \Gm{{\bbG_m}}
\def \GL{{\rm GL}}

\newcommand{\Mgt}{M_g^{\rm tr}}
\newcommand{\Agt}{A_g^{{\rm tr}}}

\newcommand{\tgt}{t_g^{\rm tr}}

\def \MMg{\mathcal{M}_g}
\def \MMgb{\ov{\mathcal{M}}_g}
\def \AAg{\mathcal{A}_g}
\def \AAgb{\ov{\mathcal{A}}_g}
\def \AAgmod{\ov{\mathcal{A}}_g^{\rm mod}}

\def \tgb{\ov{t}_g}
\def \tgt{t_g^{\rm tr}}

\def \red{{\rm red}}
\def \trop{{\rm trop}}

\def \Del{{\rm{Del}}}

\def \Null{{\rm Null}}

\def\PXgb{\overline{P^{g-1}_X}}
\def\TXg{\Theta_X}

\begin{document}

\title[Tropicalizing vs Compactifying  the Torelli morphism]{Tropicalizing vs Compactifying \\ the Torelli morphism}

\author{Filippo Viviani}
\address{ Dipartimento di Matematica,
Universit\`a Roma Tre,
Largo S. Leonardo Murialdo 1,
00146 Roma (Italy)}
\email{filippo.viviani@gmail.it}

\begin{abstract}
In this paper, we compare the compactified Torelli morphism $\tgb$ (as defined by V. Alexeev) and the tropical Torelli map $\tgt$ (as defined by the author in a joint work
with S. Brannetti and M. Melo, and furthered studied by M. Chan).
Our  aim is twofold: on one hand, we will review the construction and main properties of $\tgb$ and $\tgt$, focusing in particular on the description of their fibers achieved by
the author in joint works with L. Caporaso;  on the other hand, we will clarify the relationship between $\tgb$ and $\tgt$ via the introduction of the reduction maps and the tropicalization
maps.
\end{abstract}

\maketitle


\section{Introduction}

The Torelli morphism (or map) $t_g:\MMg\to \AAg$ (for $g\geq 2$) is the morphism from the moduli stack $\MMg$  of
connected smooth projective curves of genus $g$ to the moduli stack $\AAg$ of principally polarized (or p.p. for short) abelian varieties of dimension $g$, sending a curve into its Jacobian.
The Torelli morphism $t_g$ has played a central role
since the early developments of classical algebraic geometry since it establishes a bridge between
the two most studied moduli stacks, namely $\MMg$ and $\AAg$. We just want to mention two classical 
results on the Torelli morphism:
the Torelli theorem which says that $t_g$ is injective (on geometric points); the Schottky problem which asks for a characterization of the image of $t_g$.


In this paper, we are interested in two other maps that have been recently defined
starting from the Torelli morphism: the \textbf{compactified Torelli morphism}  $\tgb$ (defined in \cite{Ale2} based upon the work \cite{Ale1}) and the \textbf{tropical Torelli map} $\tgt$ (defined in 
\cite{BMV} based upon the work \cite{CV1}, and furthered studied in \cite{Cha}).
Our  aim is twofold: on one hand, we will review the definitions and main properties of the above maps, focusing in particular on the description of their fibers achieved in
\cite{CV1} and \cite{CV2};   on the other hand, we will explain and better clarify the relationship between $\tgb$ and $\tgt$.
Before we can state the results of this paper, we need to briefly recall the definitions of $\tgb$ and $\tgt$.


The moduli stacks $\MMg$ and $\AAg$ are irreducible and separated but not proper. However, they both admit a modular compactification: $\MMg$ is an open and dense substack of the
Deligne-Mumford \cite{DM} moduli stack $\MMgb$ of stable curves (which we review in Subsection \ref{S:stablecurves}); $\AAg$ is an open and dense substack of the main irreducible
component $\AAgb$ of the Alexeev \cite{Ale1} moduli stack of p.p. stable semi-abelic pairs (which we review in Subsection \ref{S:ppSSAP}).

V. Alexeev has shown in \cite{Ale2} that the Torelli morphism $t_g$ extends to a modular morphism $\ov{t}_g:\MMgb\to \AAgb$, called the \emph{compactified Torelli morphism}, which sends
a stable curve  into its compactified Jacobian of degree $g-1$; see Subsection \ref{S:comp-Tor} for more details.
The fibers of the compactified Torelli morphism $\tgb$ have been described by Caporaso-Viviani in \cite{CV2}.
In particular, it turns out that $\tgb$ is injective on the open subset of $\MMgb$ consisting of stable
curves that do not have separating edges nor separating pairs of edges or, in other words, stable curves whose dual graph is $3$-edge-connected (see Corollary \ref{C:inj-tgb}).
We recall the precise description of the fibers of $\tgb$ in Subsection
\ref{S:fibers-tgb}, where we also point out a new interesting relationship with the image of the canonical
morphism (see Theorem \ref{T:Tor-canonical}).

On the tropical side of the picture, Brannetti-Melo-Viviani \cite{BMV} (based on the work of Caporaso-Viviani in \cite{CV1}) constructed
the moduli space $\Mgt$ of tropical curves of genus $g$ and the  moduli space $\Agt$ of tropical p.p.   abelian varieties of dimension $g$. The spaces $\Mgt$ and $\Agt$ are constructed in loc. cit. as stacky fans, i.e. connected topological spaces obtained by gluing in a suitable way cones quotiented out by finite automorphism groups.
In particular, $\Mgt$ and $\Agt$ are endowed with a natural topology, which we call the  euclidean topology,
with respect to which they are Hausdorff spaces, as proved by L. Caporaso \cite{Cap1} for $\Mgt$
and by M. Chan \cite{Cha} for $\Agt$. We refer the reader to Subsection \ref{S:trop-curves} for more details
on $\Mgt$ and to Subsection \ref{S:tropabvar} for more details on $\Agt$.

In \cite{BMV}, the authors also construct a map $\tgt:\Mgt\to \Agt$, called the \emph{tropical Torelli map}, which sends a tropical curve into its tropical Jacobian. In loc. cit., it is shown that $\tgt$ is a
map of stacky fans; in particular, it is a continuous map. The fibers of the tropical Torelli map $\tgt$ have been described by Caporaso-Viviani in \cite{CV1}. In particular, it turns out that $\tgt$ is
injective on the locally closed subset of $\Mgt$ consisting of tropical curves whose underlying graph is $3$-vertex-connected and has genus $g$ (see Corollary \ref{C:inj-tgt}).
We recall the precise description of the fibers of $\tgt$ in Subsection \ref{S:fibers-tgt} (see Fact \ref{F:tropTorfibers}).

The main motivation of this work is  the following natural

\vspace{0,2cm}

{\bf Question}:
\emph{What is the relationship between the compactified Torelli morphism $\tgb$ and the tropical Torelli map
$\tgt$?}

\vspace{0,2cm}



In order to answer the above question, let us fix
a complete DVR (= discrete valuation ring) $R$ with maximal ideal $\m$ and assume that its residue field $k:=R/\m$ is
algebraically closed.
Let $K$ be the quotient field of $R$ and $\val: K\to \bbZ\cup \{\infty\}$ the associated valuation.
Note that the valuation $\val$ induces a topology on $K$, which is called the non-archimedean topology (see Notations \ref{S:notations}).
The sets $\MMg(K)$ and $\AAg(K)$ of $K$-valued points of, respectively, $\MMg$ and $\AAg$
inherit a topology from the topology on $K$, which we also call non-archimedean topology.
The classical Torelli morphism $t_g:\MMg\to \AAg$ induces a continuous map $\MMg(K)\to \AAg(K)$ which, by a slight abuse of notation, we also denote by  $t_g$.

On the other hand, the compactified Torelli morphism $\tgb$ induces a map $\MMgb(k)\to \AAgb(k)$ between
the $k$-valued points of, respectively, $\MMgb$ and $\AAgb$ (which we also denote by $\tgb$ by a slight abuse
of notation); moreover this map is continuous with respect to the Zariski topologies on $\MMgb(k)$ and on
$\AAgb(k)$.

After these preliminaries, we can state the first main result of this note, which can be seen as an answer to the above Question.

\begin{teoalpha}\label{T:mainthm1}
There is a commutative diagram of sets
\begin{equation}\label{E:fund-diag}
\xymatrix{
\MMgb(k)\ar[d]^{\tgb} &  \MMg(K) \ar[d]^{t_g} \ar[l]_{\rm red} \ar[r]^{\rm trop} & \Mgt \ar[d]^{t_g^{\rm tr}}\\
\AAgb(k)  & \AAg(K)  \ar[l]_{\rm red} \ar[r]^{\rm trop} & \Agt\\
}\end{equation}
where the vertical maps are continuous if we put the Zariski topology on the spaces on the left hand side, the non-archimedean topology on the spaces in the middle and 
the euclidean topology on the spaces on the right hand side.
\end{teoalpha}

The spaces appearing on the left and on the right hand sides of the above diagram \eqref{E:fund-diag} admit natural
stratifications into locally closed subsets and our second main result concerns  the compatibility of these stratifications
with respect to the above reduction maps $\red$ and tropicalization maps $\trop$. Let us briefly review how these
stratifications are defined.

On one hand, to every stable weighted graph $(\Gamma,w)$ of genus $g$, we associate two locally closed subsets
\begin{equation*}
\begin{aligned}Ê
\MMgb(\Gamma,w)\subset \MMgb(k), \\
\Mgt(\Gamma,w)\subset \Mgt,
\end{aligned}
\end{equation*}
where $\MMgb(\Gamma,w)$ consists of all stable curves $X\in \MMgb(k)$ whose dual graph is $(\Gamma,w)$ and $\Mgt(\Gamma,w)$ consists of all tropical curves $C\in \Mgt$
whose underlying combinatorial type is $(\Gamma,w)$.
As observed already in \cite[Sec. 6.3]{BMV}, the above stratifications of $\MMgb(k)$ and of $\Mgt$ enjoy the following duality property with respect to the inclusions among the closures of strata:
\begin{equation}\label{E:dual1}
\MMgb(\Gamma,w)\subseteq \ov{\MMgb(\Gamma',w')}\quad \Longleftrightarrow \quad \ov{\Mgt(\Gamma,w)}\supseteq \Mgt(\Gamma',w').
\end{equation}
We refer the reader to Section \ref{S:modcurves} for more details.

On the other hand, to every equivalence class of Delaunay decompositions $[\Delta]$ of $\bbR^g$, we associate two locally closed subsets
\begin{equation*}
\begin{aligned}Ê
\AAgb([\Delta])\subset \AAgb(k), \\
\Agt([\Delta])\subset \Agt,
\end{aligned}
\end{equation*}
where $\AAgb([\Delta])$ consists of all p.p. stable semi-abelic pairs whose associated Delaunay decomposition is $[\Delta]$ and $\Agt([\Delta])$ consists of all tropical p.p. abelian varieties
whose associated Delaunay decomposition is $[\Delta]$. Also the stratifications of $\AAgb(k)$ and of $\Agt$ enjoy a similar duality property:
\begin{equation}\label{E:dual2}
\AAgb([\Delta])\subseteq \ov{\AAgb([\Delta'])}\quad \Longleftrightarrow \quad \ov{\Agt([\Delta])}\supseteq \Agt([\Delta']).
\end{equation}
We refer the reader to Section \ref{S:abvar} for more details.


\begin{teoalpha}\label{T:mainthm3}
\noindent
\begin{enumerate}[(i)]
\item \label{T:mainthm3i} For any stable weighted graph $(\Gamma,w)$ of genus $g$, it holds that
$$\red^{-1}(\MMgb(\Gamma,w))=\trop^{-1}(\Mgt(\Gamma,w)).$$
\item \label{T:mainthm3ii} For any equivalence class of Delaunay decompositions $[\Delta]$ of $\bbR^g$, it holds that
$$\red^{-1}(\AAgb([\Delta]))=\trop^{-1}(\Agt([\Delta])).$$
\end{enumerate}
\end{teoalpha}

\vspace{0,2cm}

We have observed in Theorem \ref{T:mainthm1} that the vertical maps appearing in diagram \eqref{E:fund-diag}
are continuous with respect to the topologies specified in loc. cit.
What about the continuity properties of the reduction maps $\red$ and of the tropicalization maps $\trop$?
For the reduction maps $\red$, we can prove the following


\begin{teoalpha}\label{T:mainthm2}
In the diagram \eqref{E:fund-diag}, the reduction maps $\red$ are anticontinuous (i.e. the inverse image of a closed subset is open).
\end{teoalpha}

Indeed, we prove more generally that  for any stack $\calX$ proper over $\Spec R$
the natural reduction map $\calX(K)\to \calX(k)$ is anticontinuous (see Corollary \ref{C:red-proper}).

As far as the tropicalization maps $\trop$ are concerned, we make the following

\vspace{0,2cm}

\noindent
\textbf{Conjecture.}
\emph{In the diagram \eqref{E:fund-diag}, the tropicalization maps $\trop$ are continuous.}

\vspace{0,2cm}

Note that a positive answer to the above Conjecture, together with Theorem \ref{T:mainthm2},
could be regarded as a conceptual explanation of the duality \eqref{E:dual1} among the stratifications of
$\Mgt$ and of $\MMgb(k)$ as well as of the duality \eqref{E:dual2} among the stratifications of
$\Agt$ and of $\AAgb(k)$.

\vspace{0,2cm}

While this work was been completed, we heard from \cite{Rab} of a work in progress of M. Baker and J. Rabinoff, where 
they will prove the commutativity of the right square of diagram \eqref{E:fund-diag} in greater generality, namely 
working over an arbitrary non-archimedean (not necessarily discrete) valued field $K$ and replacing the topological spaces in
the middle with the (bigger) Berkovich analytifications of $\MMg$ and of $\AAg$.

Moreover, while this paper was under the refereeing process, the interesting preprint \cite{ACP} by D. Abramovich, L. Caporaso and S. Payne was posted on arXiv. 
In \cite[Thm. 1.2.1(1)]{ACP}, the authors prove that the compactification $\ov{\Mgt}$ of  $\Mgt$ constructed by L. Caporaso in \cite[\S 3.3]{Cap1} is isomorphic to the  skeleton of the 
Berkovich analytification $\MMgb^{\rm an}$ of $\MMgb$. 
Moreover, they show in \cite[Thm. 1.2.2(2)]{ACP} that the tropicalization map $\trop: \MMg(K)\to \Mgt$ of Theorem \ref{T:mainthm1} extends to a continuous, proper and surjective map 
${\rm Trop}:\MMgb^{\rm an}\to \ov{\Mgt}$.
As a corollary, one gets that the tropical map $\trop: \MMg(K)\to \Mgt$ is continuous, thus providing a positive partial answer to the above Conjecture.

\vspace{0,2cm}

We conclude this introduction with an outline of the paper and with the notations we are going to use throughout the paper.

\begin{nota}{\emph{Outline of the paper}}
\label{S:outline}

In Section \ref{S:modcurves}, we review the definition and main properties of the moduli stack $\MMgb$ of Deligne-Mumford stable curves (Subsection \ref{S:stablecurves}) and of the moduli
space $\Mgt$ of tropical curves  (Subsection \ref{S:trop-curves}).   Moreover, we define the reduction map curves in \ref{N:redmap1}, the tropicalization map for curves
in \ref{N:tropmap1}, and we prove Theorem \ref{T:mainthm3}\eqref{T:mainthm3i} at the end of the section.

In Section \ref{S:abvar}, we review the definition and main properties of the moduli space $\Agt$ of tropical p.p. abelian varieties (Subsection \ref{S:tropabvar}) and of the main component
$\AAgb$ of the moduli stack  of Alexeev  p.p.  stable semiabelic pairs (Subsection \ref{S:ppSSAP}).  Moreover, we define the tropicalization map for abelian varieties
in \ref{N:trop-Ag}, the reduction map for abelian varieties in \ref{N:redmap2} and we prove Theorem \ref{T:mainthm3}\eqref{T:mainthm3ii} at the end of the section.

Section \ref{S:Tormaps} is devoted to the two Torelli maps $\tgt$ and $\tgb$. In Subsection \ref{S:trop-Tor}, we review the definition of the tropical Torelli map $\tgt$ and we prove the second half
 of Theorem \ref{T:mainthm1}, i.e. that the tropicalization maps commute with the Torelli maps (see Theorem \ref{T:commuta2}). Moreover, we recall the description obtained
 in \cite{CV1} of the fibers of $\tgt$ in  \ref{S:fibers-tgt}.  In Subsection \ref{S:comp-Tor}, we review the definition of the compactified Torelli morphism $\tgb$ and we prove the first half
 of Theorem \ref{T:mainthm1}, i.e. that the reduction maps commute with the Torelli maps (see Theorem \ref{T:commuta1}). Moreover, we recall the description obtained
in \cite{CV2} of the fibers of $\tgb$ in  \ref{S:fibers-tgb}.   Section \ref{S:Tormaps} ends with a new description of the fibers of $\tgb$ on the locus of curves free from separating nodes
and not hyperelliptic in terms of their canonical morphisms (see Theorem \ref{T:Tor-canonical}).

Finally, in Section \ref{S:anti-redmap}, we prove Theorem \ref{T:mainthm2}, i.e. the anticontinuity of the reduction maps. Indeed, we show that the same result is true
for any stack proper over $\Spec R$ (see Corollary \ref{C:red-proper}).

\end{nota}

\begin{nota}{\emph{Notations}}
\label{S:notations}

\begin{itemize}

\item Throughout the paper, we fix an integer $g\geq 2$.

\item We fix a complete \footnote{Indeed, everything that we are going to say in this paper can be extended to a strictly Henselian discrete valuation ring $R$.
However, for simplicity, we assume that $R$ is complete.}  discrete valuation ring (DVR for short) $R$ with maximal ideal $\m$ and we assume that its residue field $k:=R/\m$ is algebraically
closed. Given an element $x\in R$, we denote by $\ov{x}\in k$ its reduction modulo the maximal ideal $\m$. Let $K$ be the quotient field of $R$.

\item We denote by $s$ the closed (or special) point of $\Spec R$ and by $\eta$ its generic point. In particular, the residue field of $s$ is equal to $k$ while the residue field of $\eta$ is equal to $K$.

\item Let $\val: K\to \bbZ\cup \{\infty\}$ be the valuation associated to the discrete valuation ring $R$.  The valuation $\val$ induces a non-archimedean norm $|.|$ on $K$
defined as
$$|x|:=e^{-\val(x)},$$
where $e$ is the Euler number (indeed for the purpose of what follows we can replace $e$ by any positive real number). The norm $|.|$ induces  a metric $d$ on $K$ defined by
$$d(x,y)=|x-y|.$$
The topology on $K$ induced by this metric $d$ is called the \emph{non-archimedean topology} on $K$. We endow $R\subset K$ with the subspace topology, which is
called the non-archimedean topology on $R$. Note that the maximal ideal $\m\subset R$ coincides with the open ball of radius $1$ centered at $0$:
$$\m=\{x\in R\: :\: |x| <1\}=\{x\in R\: : \: d(0,1)<1\}.$$
Similarly, the product topologies on $R^n$ and $K^n$ are called non-archimedean topologies.

\item Given any finite extension of fields $K\subseteq K'$, the valuation $\val$ on $K$ can be extended in a unique way to a valuation $\val'$ on $K'$ (using the fact that $K$ is complete
with respect to $\val$). The valuation ring of $\val'$, also called the valuation ring of $K'$ and denoted by  $R'$, is also equal to the integral closure of $R$ in the field $K'$.
Note that the valuation ring $R'$ is also a complete DVR. With a slight abuse of notation, we denote by $s$ the special point of $\Spec R'$ and by $\eta$ the generic point of
$\Spec R'$.

\item A map $f:X\to Y$ between topological spaces is said to be \emph{anticontinuous} if the inverse image of any closed
subset of $Y$ is open in $X$, or equivalently if the inverse image of any open subset of $Y$ is closed in $X$.

\end{itemize}
	
\end{nota}

\section{Moduli spaces of curves}\label{S:modcurves}

\subsection{The moduli stack $\MMgb$ of stable curves}\label{S:stablecurves}

The moduli stack $\MMg$ of connected smooth projective curves of genus $g\geq 2$ can be compactified by adding stable curves.

\begin{defi}\label{D:stablecurves}
A \emph{stable curve} $X$ of genus $g$ over a field $k$ is a connected projective nodal curve over  $k$ of arithmetical genus $g$ whose  canonical sheaf $\omega_X$ is ample.
\end{defi}

The following celebrated result is due to Deligne-Mumford \cite{DM}.

\begin{fact}[Deligne-Mumford]\label{F:DM-comp}
The  stack $\MMgb$ of stable curves of genus $g$ is proper and smooth over $\Spec \bbZ$. Moreover, $\MMgb$ is irreducible of dimension $3g-3$ and it contains $\MMg$ as a dense open substack.
\end{fact}

\begin{nota}{\emph{The stratification of $\MMgb(k)$}}
\label{N:strataMMgb}

The set $\MMgb(k)$ of all stable curves of genus $g$ defined over $k$ endowed with its Zariski topology
admits a stratification into locally closed subspaces, parametrized by stable weighted graphs  of genus $g$, whose definition we recall below.

\begin{defi}\label{D:mark-graph}
A \emph{weighted graph} is a couple $(\Gamma,w)$ consisting of a finite connected
graph $\Gamma$ (possibly with loops or parallel edges) and a function $w:V(\Gamma)\to \bbN$, called
the weight function.
A weighted graph  is called \emph{stable} if any vertex $v$ of weight zero
(i.e. such that $w(v)=0$) has valence $\val(v)\geq3$.
The total weight of a weighted graph  $(\Gamma,w)$ is
$$|w|:=\sum_{v \in V(\Gamma)} w(v),$$
and the genus of $(\Gamma,w)$ is equal to
$$g(\Gamma,w):=g(\Gamma)+|w|.$$
\end{defi}

Given a weighted graph $(\Gamma, w)$, the automorphism group $\Aut(\Gamma, w)$ of  $(\Gamma,w)$ consists of all the pairs 
$(\sigma, \psi)$ where $\sigma$ is a permutation of the vertices $V(\Gamma)$ of $\Gamma$ and $\psi$ is a permutation 
of the edges $E(\Gamma)$ of $\Gamma$ such that:
\begin{itemize}
\item $w(\sigma(v))=w(v)$ for any $v\in V(\Gamma)$;
\item if an edge $e\in E(\Gamma)$ is incident to a vertex $v\in V(\Gamma)$ then $\psi(e)$ is incident to $\sigma(v)$. 
\end{itemize}


To every stable curve $X$ of genus $g$ it is naturally associated a stable weighted graph  of genus $g$, called its dual weighted graph,
which captures the combinatorics of the stable curve.

\begin{defi}\label{D:dualgraph}
The \emph{dual weighted graph} of a stable curve $X$ of genus $g$ is the weighted graph $(\Gamma_X,w_X)$ defined as it follows.

The vertices $V(\Gamma_X)$ of the graph $\Gamma_X$ are in bijection with the irreducible components of $X$ while the edges $E(\Gamma_X)$ of $X$
are in bijection with the nodes of $X$. An edge $e\in E(\Gamma_X)$ corresponding to a node $n_e$ of $X$ links  the (possibly equal) vertices $v_1$ and $v_2$
corresponding to the (possibly equal) irreducible components $C_{v_1}$ and $C_{v_2}$ which contain the node $n_e$.

 The weight function $w_X:V(\Gamma_X)\to \bbN$
assigns to every vertex $v$ of $\Gamma_X$ the geometric genus of the irreducible component $C_v$ corresponding to the vertex $v$.
\end{defi}

It is easy to check that the dual weighted graph $(\Gamma_X,w_X)$ of a stable curve $X$ of genus $g$ is stable and of genus $g$.

\vspace{0,2cm}

To every stable weighted graph $(\Gamma, w)$ of genus $g$, we associate the following subset of $\MMgb(k)$:
\begin{equation}\label{E:strataMMgb}
\MMgb(\Gamma,w):=\{X\in \MMgb(k) \: : \: (\Gamma_X,w_X)=(\Gamma,w)\}.
\end{equation}
As $(\Gamma,w)$ varies among all stable weighted graphs of genus $g$, we get a stratification of $\MMgb(k)$ into disjoint locally closed subsets.
In order to describe the inclusion relations between the closures of these strata, we introduce the following order relation among all stable weighted graphs of genus $g$.

\begin{defi}\label{D:spec-graphs}
Given two weighted graphs $(\Gamma, w)$ and $(\Gamma', w')$, we say that $(\Gamma,w)$ \emph{dominates} $(\Gamma',w')$, and we write $(\Gamma,w)\geq (\Gamma',w')$, if
$\Gamma'$ is obtained from $\Gamma$ by contracting some of its edges and the weight function $w'$ is obtained from the weight function $w$ by an iteration of the following rule:
\begin{itemize}
\item If $\Gamma'$ is obtained from $\Gamma$ by contracting an edge $e$ that joins two distinct vertices $v_1$ and $v_2$, then the vertex $\wt{v}$ of $\Gamma'$ which is the image
of the two vertices $v_1$ and $v_2$ has weight $w'(\wt{v})=w(v_1)+w(v_2)$.
\item If $\Gamma'$ is obtained from $\Gamma$ by contracting a loop $e$ around the vertex $v$, then the vertex $\wt{v}$ of $\Gamma'$ which is the image
of $v$ has weight $w'(\wt{v})=w(v)+1$.
\end{itemize}
\end{defi}
It is easy to see that if $(\Gamma,w)\geq (\Gamma',w')$ then $g(\Gamma,w)=g(\Gamma',w')$ and moreover, if $(\Gamma,w)$ is stable, then $(\Gamma',w')$ is stable.

We can now describe the inclusion relation among the closures of the strata of $\MMgb$.

\begin{fact}\label{F:strataMgb}
The space $\MMgb(k)$ admits a stratification into disjoint locally closed subsets
$$\MMgb(k)=\coprod_{(\Gamma,w)} \MMgb(\Gamma,w),$$
as $(\Gamma,w)$ varies among all stable weighted graphs of genus $g$.

Given two stable weighted graphs $(\Gamma,w)$ and $(\Gamma',w')$ of genus $g$, we have that
$$\MMgb(\Gamma,w)\subseteq \ov{\MMgb(\Gamma',w')}\Leftrightarrow (\Gamma,w)\geq (\Gamma',w').$$
\end{fact}
\begin{proof}
This is well-known, see e.g. \cite[Chap. XII, Sec. 10]{GAC2} or \cite[\S 4.2]{Cap1}.
\end{proof}

\begin{remark}
Indeed, the stratification of the topological space $\MMgb(k)$ described in Fact \ref{F:strataMgb} is  induced by a stratification of
the stack $\MMgb$ into locally closed substacks. We refer to \cite[Chap. XII]{GAC2} for more details.
\end{remark}

\end{nota}

\begin{nota}{\emph{The reduction map $\red: \MMg(K)\to \MMgb(k)$}}
\label{N:redmap1}

We are now ready to define the reduction map $\red: \MMg(K)\to \MMgb(k)$ appearing in the diagram \eqref{E:fund-diag}.
Since the stack $\MMgb$ is proper, the valuative criterion of properness for stacks gives that for any map $f:\Spec K\to \MMg\subseteq \MMgb$ there exists
a finite extension $K'$ of $K$ with valuation ring $R'$ and a unique map $\phi:\Spec R'\to \MMgb$ such that the following diagram is commutative
$$
\xymatrix{
\Spec R'  \ar[rrrd]^{\phi}& & & \\
\Spec K' \ar[r] \ar[u] & \Spec K \ar[r]^f & \MMg \ar@{^{(}->}[r]& \MMgb.
}$$
In other words, given a connected smooth projective curve $X\in \MMg(K)$, up to a finite extension $K\subseteq K'$ with valuation ring $R'$,
there exists a unique family of stable curves $\calX'\to \Spec R'$, called the \emph{stable reduction} of $X$ with respect to the extension $K\subseteq K'$, such that its generic fiber
$\calX'_{\eta'}:=\calX\times_{\Spec R'} \Spec K'$ is isomorphic to $X_{K'}:=X\times_K K'$. Note that the residue field of $R'$ is equal to $k$, since $k$ was assumed to be algebraically closed.

\begin{lemmadefi}\label{D:red-curves}
The reduction map
$$\red: \MMg(K)\to \MMgb(k)$$
is defined by sending $X\in \MMg(K)$ to the central fiber $\calX'_s\in \MMgb(k)$ of a stable reduction $\calX' \to \Spec R'$ of $X$
with respect to some finite field extension $K\subseteq K'$ with valuation ring $R'$.
The isomorphism class of $\calX'_s\in \MMgb(k)$ does not depend on the chosen field extension $K\subset K'$ and is denoted by $\red(X)$.
\end{lemmadefi}
\begin{proof}
Let $K'$ and $K''$ two finite field extensions of $K$, with valuation rings respectively $R'$ and $R''$, such that $X$ admits a stable reduction $\calX'\to \Spec R'$ with respect to $K'$ and
a stable reduction $\calX''\to \Spec R''$ with respect to $K''$.
Choose an algebraic closure $\ov{K}$ of $K$ that contains $K'$ and $K''$ and consider, inside $\ov{K}$, the smallest field extension $K\subseteq L$ that contains $K'$ and $K''$.
Clearly, $L$ is a finite field extension of $K$ and we denote by $S$ its valuation ring. The base change of each of the two families $\calX'\to \Spec R'$ and $\calX''\to \Spec R''$ to
$\Spec S$ is clearly a stable reduction with respect to the extension $K\subseteq L$.  By the uniqueness of the stable reduction, these two pull-backs must be isomorphic and in  particular
their central fibers must be isomorphic. However, since $k$ is assumed to be algebraically closed, the central fibers of these two pull-backs are equal to $\calX'_s$ and $\calX''_s$; hence
we must have that $\calX'_s\cong \calX''_s$, q.e.d.
\end{proof}

\end{nota}

\subsection{The moduli space $\Mgt$ of tropical curves}\label{S:trop-curves}

Recall the definition of tropical curves introduced in \cite{BMV}, generalizing slightly the original definition of Mikhalkin-Zharkov in \cite{MZ}.

\begin{defi}\label{D:trop-curves}
A \emph{tropical curve} $C$ is the datum of a triple $(\Gamma,w,l)$
consisting of a stable weighted graph  $(\Gamma,w)$, called the
combinatorial type of $C$, and a function $l:E(\Gamma)\to \bbR_{>0}$,
called the length function.
The genus $g(C)$ of $C$ is the genus of its combinatorial type.
\end{defi}

Given a stable weighted graph  $(\Gamma,w)$ of genus $g$, we define $\Mgt(\Gamma,w)$ to be the set of tropical curves of combinatorial type equal to $(\Gamma,w)$.
Note that a tropical curve $C\in \Mgt(\Gamma,w)$ is determined by a length function $l:E(\Gamma)\to \bbR_{>0}$. However, different length functions can give rise to the same
tropical curve if they differ by an automorphism of the weighted graph  $(\Gamma,w)$. 
Therefore, we have a natural identification
\begin{equation}\label{E:strataMgt}
\Mgt(\Gamma,w)=\bbR_{>0}^{E(\Gamma)}/\Aut(\Gamma,w).
\end{equation}
These spaces can indeed by glued together along their boundaries in order to obtain a topological space $\Mgt$, called the moduli space of tropical curves of genus $g$,
whose points are in bijection with tropical curves of genus $g$ (see \cite{BMV}). In loc. cit., the space $\Mgt$ is endowed with the structure of a stacky fan. Here, for simplicity, we treat it simply
as a topological space. We summarize all the known properties of $\Mgt$ in the following

\begin{fact}
\label{F:Mgt}
\noindent
\begin{enumerate}[(i)]
\item \label{F:Mgt1} There exists a topological space $\Mgt$ whose points are in natural bijection with tropical curves of genus $g$. Moreover, the topological space $\Mgt$ is 
normal (hence Hausdorff), locally compact, paracompact, locally contractible, metrizable and second countable.
\item \label{F:Mgt2}  The topological space $\Mgt$ admits a stratification into disjoint locally closed subsets
$$\Mgt=\coprod_{(\Gamma,w)} \Mgt(\Gamma,w),$$
as $(\Gamma,w)$ varies among all stable weighted graphs of genus $g$.
\item  \label{F:Mgt3} Given two stable weighted graphs $(\Gamma,w)$ and $(\Gamma',w')$ of genus $g$, we have that
$$\ov{\Mgt(\Gamma,w)}\supseteq \Mgt(\Gamma',w')\Leftrightarrow (\Gamma,w)\geq (\Gamma',w').$$
\end{enumerate}
\end{fact}
\begin{proof}
The topological space $\Mgt$ has been constructed in \cite{BMV} and further studied in \cite{Cap1}. Properties \eqref{F:Mgt2} and \eqref{F:Mgt3} follows from \cite[\S 3]{BMV}. 
The topological properties of $\Mgt$ stated in \eqref{F:Mgt1}Ê are proved in \cite[\S 2]{CMV}.
\end{proof}

\begin{nota}{\emph{The tropicalization map $\trop: \MMgb(K)\to \Mgt$}}
\label{N:tropmap1}

We are now ready to define the tropicalization map $\trop:\MMg(K)\to \Mgt$ appearing in the diagram \eqref{E:fund-diag}.
Given a connected projective smooth curve $X$ over $K$, consider a finite field extension
$K\subseteq K'$ with valuation ring $R'$ such that  the base change $X_{K'}$ of $X$ to $K'$ admits a stable reduction $\calX'\to \Spec R'$, in the sense of \ref{N:redmap1}.
Consider now  a node $n$ of the central fiber $\calX'_{s}$ of $\calX' \to  \Spec R'$. Since the generic fiber of $\calX'$ is smooth, by the deformation theory of nodal singularities,
it follows easily that a local equation of the surface $\calX'$ at $n$ can be chosen to be $xy=(t')^{w_n}$,  where $t'$ is some fixed uniformizer of $R'$ (i.e. a generator of the maximal ideal $\m'$ of
$R'$) and $w_n\in \bbZ_{>0}$ is some uniquely determined natural number, which we call the \emph{width of the node} $n$.

\begin{lemmadefi}\label{D:tropmap-curves}
The tropicalization map
$$\trop: \MMg(K)\to \Mgt$$
is defined by sending $X\in \MMg(K)$ into the tropical curve $C'\in \Mgt$ such that:
\begin{itemize}
\item the combinatorial type of $C'$ is given by the dual weighted graph $(\Gamma_{\calX'_{s}},w_{\calX'_{s}})$ of the special fiber
$\calX'_{s}$ of the stable reduction of $X$ with respect to some finite field extension $K\subseteq K'$;
\item the length of an edge $e\in E(\Gamma_{\calX'_{s}})$ is equal to $l'(e):=\displaystyle \frac{w_{n_e}}{[K':K]}$, where $w_{n_e}$ is the width of the node $n_e$ of $\calX'_{s}$ corresponding
to the edge $e$ and $[K':K]$ is the degree of the finite field extension $K\subseteq K'$.
\end{itemize}
The so defined tropical curve $C'\in \Mgt$  does not depend on the chosen  field extension $K\subseteq K'$ and is denoted by $\trop(X)$.
\end{lemmadefi}
\begin{proof}
Let $K'$ and $K''$ two finite field extensions of $K$, with valuation rings respectively $R'$ and $R''$, such that $X$ admits a stable reduction $\calX'\to \Spec R'$ with respect to $K'$ and
a stable reduction $\calX''\to \Spec R''$ with respect to $K''$. Denote by $C'$ and $C''$ the tropical curves associated to the stable reductions $\calX'$ and $\calX''$ according to the
above described procedure.

As in the proof of Lemma-Definition \ref{D:red-curves}, we can find a finite field extension $K\subseteq L$, with valuation ring $S$,  that contains $K'$ and $K''$ as subfields.
In the proof of loc. cit., we have shown that the special fibers $\calX'_{s}$ and $\calX''_{s}$ of the two stable reductions $\calX'$ and $\calX''$ are isomorphic.
This implies that the combinatorial types of $C'$ and of $C''$ are the same. It remains to show that the length function $l'$ on $C'$ coincides with the length function $l''$ on $C''$.
Consider now an edge $e\in E(\Gamma_{\calX'_{s}})=E(\Gamma_{\calX''_{s}})$ and its corresponding
node $n:=n_e\in \calX'_{s}=\calX''_{s}$. If we choose a uniformizer $z$ for $S$, then $t':=z^{[L:K']}$ is a uniformizer for $R'$ and $t'':=z^{[L:K'']}$ is a uniformizer for $R''$.
Therefore, if the local equation of $\calX'$ (resp. $\calX''$) at $n$ is given by $xy=(t')^{w'_{n}}$ (resp. $xy=(t'')^{w''_n}$) then the local equation of the surface
$\calX'\times_{\Spec R'} \Spec S$ (resp. $\calX''\times_{\Spec R''} \Spec S$) at $n$ is given by $xy=z^{w'_n[L:K']}$ (resp. $xy=z^{w''_n[L:K'']}$).
Since  $\calX'\times_{\Spec R'} \Spec S\cong \calX''\times_{\Spec R''} \Spec S$ by the uniqueness of the stable reduction, we get that
$w'_n[L:K']=w''_n[L:K'']$. This implies that
$$l'(e)=\frac{w'_{n}}{[K':K]}=\frac{w'_{n}[L:K']}{[L:K]}=\frac{w''_{n}[L:K'']}{[L:K]}=\frac{w''_{n}}{[K'':K]}=l''(e),$$
which shows that $l'$ is equal to $l''$, q.e.d.
\end{proof}

\begin{remark}\label{R:tropmap-curves}
Given a curve $X$ over $K$, the metrized graph underlying the tropical curve $\trop(X)$ associated to $X$ in Lemma-Definition \ref{D:tropmap-curves} is the reduction (metrized) graph 
of $X$ as defined in \cite[p. 9-10]{CR}. 
Moreover, $\trop(X)$ is the minimal skeleton in the Berkovich analytification $X^{\rm an}$  of $X$, see \cite[Cor. 5.50]{BPR}.

\end{remark}

Now that we have defined the reduction map $\red:\MMg(K)\to \MMgb(k)$ and the tropicalization map $\trop:\MMg(K)\to \Mgt$, it is easy to prove the first half of Theorem \ref{T:mainthm3}.

\begin{proof}[Proof of Theorem \ref{T:mainthm3}\eqref{T:mainthm3i}]
By comparing Lemma-Definition \ref{D:red-curves} and Lemma-Defi\-nition \ref{D:tropmap-curves}, one easily realizes that, for a smooth curve $X\in \MMg(K)$, the combinatorial type
of the tropical curve $\trop(X)\in \Mgt$ is equal to the dual weighted graph of the stable curve $\red(X)\in \MMgb(k)$.
The conclusion now follows.
\end{proof}

\end{nota}

\section{Moduli spaces of abelian varieties}\label{S:abvar}

\subsection{The moduli space $\Agt$ of tropical  p.p.  abelian varieties}\label{S:tropabvar}

Recall the definition of a tropical principally polarized abelian variety introduced in  \cite{BMV}, generalizing slightly the original definition of Mikhalkin-Zharkov in \cite{MZ}.

\begin{defi}\label{D:trop-abvar}
A \emph{tropical p.p.} (= principally polarized) \emph{abelian variety} $A$ of dimension
$g$ is a pair $(\bbR^g/\Lambda, Q)$ consisting of a $g$-dimensional real torus $\bbR^g/\Lambda$ (for  a rank-$g$ lattice $\Lambda\subset \bbR^g$)
and  $Q$ is a positive semi-definite quadratic form on $\bbR^g$ such that the null space
$\Null(Q)$ of $Q$ is defined over $\Lambda\otimes \bbQ$, i.e. it admits a basis with elements in $\Lambda\otimes \bbQ$. Two tropical p.p. abelian varieties
$A=(\bbR^g/\Lambda,Q)$ and $A'=(\bbR^g/\Lambda',Q')$ are isomorphic if there exists
$h\in GL(g,\bbR)$ such that $h(\Lambda)=\Lambda'$ and $hQh^t=Q'$.
\end{defi}

Indeed, tropical p.p. abelian varieties up to isomorphism are the same thing as positive semi-definite quadratic forms
up to arithmetic equivalence, as shown in the following

\begin{remark}\label{R:equiv-abvar}
\noindent
\begin{enumerate}[(i)]
\item Every tropical p.p. abelian variety $A=(\bbR^g/\Lambda,Q)$ can be written in the form
$(\bbR^g/\bbZ^g,Q')$. In fact, it is enough to consider $Q'=hQh^t$, where
$h\in GL(g,\bbR)$ is such that $h(\Lambda)=\bbZ^g$.
\item $(\bbR^g/\bbZ^g,Q)\cong (\bbR^g/\bbZ^g,Q')$ if and
only if there exists $h\in \GL_g(\bbZ)$ such that $Q'=h Q h^t$, i.e., if and only if
$Q$ and $Q'$ are arithmetically equivalent.
\end{enumerate}
\end{remark}

Before stating the main properties of the moduli space $\Agt$ of tropical p.p. abelian varieties, we need
a digression into Delaunay decompositions of $\bbR^ g$.

\begin{defi} \label{D:paving}
\noindent
\begin{enumerate}[(i)]
\item A $\bbZ^g$-periodic integral {\it paving}
(or face-fitting decomposition)
of $\bbR^g$ of maximal rank $g$  is a set $\Delta$ of integral
polytopes $\omega\subset \bbR^g$ satisfying:
\begin{enumerate}[(a)]
\item $\bbR^g=\cup_{\omega\in \Delta} \omega$;
\item Any face of $\omega\in \Delta$ belongs to $\Delta$;
\item For any $\omega$, $\omega'\in \Delta$, the intersection
$\omega\cap \omega'$ is either empty or a common face of $\omega$ and
$\omega'$;
\item $\Delta$ is invariant by translation of $\bbZ^g$, i.e.
for any $\omega\in \Delta$ and any $h\in \bbZ^g$ the translate
$\omega+h$ belongs to $\Delta$;
\item $\#\{\omega \mod \bbZ^g\}$ is finite.
\end{enumerate}
A $\bbZ^g$-periodic integral paving
of $\bbR^g$ of rank $0\leq r\leq g$ is a set $\Delta$ of
polyhedra obtained as inverse images via a linear integral projection
$\pi:\bbR^g\to \bbR^r$ of the polytopes of
a $\bbZ^r$-periodic integral paving
$\Delta'$ of $\bbR^r$ of maximal rank.
\item Two $\bbZ^g$-periodic integral  pavings $\Delta$ and $\Delta'$ of $\bbR^g$ are \emph{equivalent} if there exists
$h\in \GL_g(\bbZ)$ such that $\Delta'$ is equal to
$$h\cdot \Delta:=\{h\cdot \omega \: :\: \omega \in \Delta\}. $$
We denote by $[\Delta]$ the equivalence class of a paving $\Delta$ of $\bbR^ g$.
\item Given two $\bbZ^g$-periodic integral  pavings $\Delta$ and $\Delta'$ of $\bbR^g$, we say that $\Delta$ is a
\emph{refinement} of $\Delta'$, and we write $\Delta\geq \Delta'$, if every polyhedron of $\Delta$ is contained in some polyhedron of $\Delta'$.

Similarly, we say that $[\Delta]$ is a refinement of $[\Delta']$, and we write $[\Delta]\geq [\Delta']$, if there exist
$h, h'\in \GL_g(\bbZ)$ such that $h\cdot \Delta \geq h'\cdot \Delta'$.
\end{enumerate}
\end{defi}

\vspace{0,2cm}

Among the $\bbZ^g$-periodic integral pavings of $\bbR^g$,
a special place is occupied by the Delaunay decompositions
associated to a positive semi-definite quadratic forms in $\bbR^g$,
whose null space is defined over $\bbQ^g$.

\begin{defi}
\label{D:Del-deco}
Let $Q$ be a positive semi-definite quadratic
form of rank $r$ in $\bbR^g$, whose null space $\Null(Q)$ is defined over $\bbQ^g$.
For any $\alpha\in \bbR^g$,  a lattice element
$x \in \bbZ^g$ is  called $\alpha$-nearest if
$$Q(x-\alpha)={\rm min}\{Q(y-\alpha)\: : \: y\in \bbZ^g\}.$$
A Delaunay cell is defined as  the
closed convex hull of all
elements of $\bbZ^g$
which are $\alpha$-nearest for some fixed $\alpha \in \bbR^g$.
Together, all the Delaunay cells form a $\bbZ^g$-periodic
integral paving of $\bbR^g$ of rank $r$, called the {\it Delaunay
decomposition} of $Q$ and denoted $\Del_Q$.
We say that a $\bbZ^g$-periodic integral paving of $\bbR^g$ is a
Delaunay paving if it is isomorphic to $\Del_Q$ for some
quadratic form $Q$ as before.
\end{defi}

Clearly, if two quadratic forms $Q$ and $Q'$ are arithmetic equivalent in the sense of Remark \ref{R:equiv-abvar}, then
their associated Delaunay decompositions $\Del_Q$ and $\Del_{Q'}$ are equivalent in the sense of Definition \ref{D:paving}.
This show that we can associate to every tropical p.p. abelian variety an equivalence class of Delaunay decompositions of $\bbR^g$.

\begin{defi}\label{D:Del-abvar}
Given a tropical p.p. abelian variety $A\cong (\bbR^g/\bbZ^g, Q)$ (see Remark \ref{R:equiv-abvar}), the Delaunay decomposition $[\Del_A]$ of $A$ is defined to be
$$[\Del_A]:=[\Del_Q].$$
\end{defi}

We are now ready to summarize the main properties of the moduli space of tropical p.p.   abelian varieties.

\begin{fact}
\label{F:Agt}
\noindent
\begin{enumerate}[(i)]
\item \label{F:Agt1} There exists a topological space $\Agt$ whose points are in natural bijection with tropical p.p.   abelian varieties of dimension $g$.
Moreover, the topological space $\Agt$ is 
normal (hence Hausdorff), locally compact, paracompact, locally contractible, metrizable and second countable.
\item \label{F:Agt2}  The topological space $\Agt$ admits a stratification into disjoint locally closed subsets
$$\Agt=\coprod_{[\Delta]} \Agt([\Delta]),$$
as $[\Delta]$ varies among all equivalence classes of Delaunay decompositions of $\bbR^g$ and
$$\Agt([\Delta]):=\{A\in \Agt\: : \: [\Del_A]=[\Delta]\}.$$
\item  \label{F:Agt3} Given two equivalence classes $[\Delta]$ and $[\Delta']$ of Delaunay decompositions of $\bbR^g$,
we have  that
$$\ov{\Agt([\Delta])}\supseteq \Agt([\Delta'])\Leftrightarrow [\Delta]\geq [\Delta'].$$
\end{enumerate}
\end{fact}
\begin{proof}
The construction of $\Agt$ and the properties \eqref{F:Agt2} and \eqref{F:Agt3}Ê can be found in \cite[\S 4]{BMV} or \cite[\S 4]{Cha}. Note that the definition of $\Agt$ given in \cite[Def. 4.2.2]{BMV}
contains a mistake that was corrected in \cite[Def. 4.9]{Cha}. The topological properties of $\Agt$ stated in \eqref{F:Agt1}Ê are proved in \cite[\S 2]{CMV}.
\end{proof}

\begin{nota}{\emph{The tropicalization map $\trop: \AAg(K)\to \Agt$}}
\label{N:trop-Ag}

We want now to define the tropicalization map $\trop: \AAg(K)\to \Agt$ appearing in the diagram \eqref{E:fund-diag}.

Recall that given an abelian variety $A$ over $K$ there is a canonical way of extending it to a separated group scheme over $\Spec R$, namely via the theory of N\'eron models.

\begin{defi}
Given an abelian variety $A$ over $K$, a \emph{N\'eron model} of $A$ over $\Spec R$ is a smooth, separated and finite type
group scheme $\calN(A)\to \Spec R$ such that its generic fiber $\calN(A)_K$ is isomorphic to $A$ and, moreover, such that
it satisfies the following universal property (called the N\'eron mapping property):
for each smooth morphism $\calY\to \Spec R$ and any $K$-morphism $\phi_K:\calY_K \to \calN(A)_K\cong A$ there exists a unique
morphism $\phi: \calY\to \calN(A)$ over $\Spec R$ extending the given morphism $\phi_K$.
\end{defi}

Clearly, the N\'eron mapping property uniquely characterizes the N\'eron model of an abelian variety $A$ over $K$, if it
exists at all. Indeed, it is a deep theorem of N\'eron that such models always exists.

\begin{fact}[N\'eron]\label{F:Nermod}
Any abelian variety over $K$ admits a N\'eron model $\calN(A)$ over $\Spec R$.
\end{fact}

For a proof, we refer the reader to the original paper of N\'eron \cite{Ner} or to the book \cite{BLR} for a modern treatment.

\vspace{0,2cm}

Recall now that given any smooth group scheme $\calA\to \Spec R$ (as for example the N\'eron model $\calN(A)$ of an abelian variety $A$ defined over $K$), there exists an open subgroup scheme $\calA^o\subseteq \calA$, called the \emph{neutral component} of $\calA$, such that the fibers of $\calA^o\to \Spec R$ are the connected components of the fibers of $\calA\to \Spec R$ which contain the identity
(see \cite[Expos\'e VIB, Thm. 3.10]{SGA3I}).

It was proved by Grothendieck (see \cite[Espos\'e IX, Thm. 3.6]{SGA7I}) that any abelian variety over $K$ is potentially semiabelian, i.e. that, after a finite extension of $K$,
the neutral component $\calN(A)^o$ of the N\'eron model of $A$ is a semiabelian scheme.

\begin{fact}[Grothendieck]\label{F:stabred}
Given an abelian variety $A$ over $K$, there exists a finite extension $K\subseteq K'$ such that the neutral component $\calN(A')^o$ of the N\'eron model $\calN(A')$ of $A':=A\times_K K'$ is a
semiabelian scheme, i.e. the special fiber $\calN(A')^o_s$ of $\calN(A')^o$ fits in a unique extension
\begin{equation}\label{E:can-ext}
0\to T' \to \calN(A')^o_s \to B'\to 0,
\end{equation}
where $T'$ is a torus over $k$ of dimension $r$ (called the \emph{rank} of $\calN(A')^o_s$) and $B'$ is an abelian variety
over $k$ of dimension $g-r$.
\end{fact}

We call the scheme $\calN(A')^o$ as above the \emph{semiabelian reduction}  of $A$ with respect to the extension
$K\subseteq K'$. It is the analogue for abelian varieties of the stable reduction for curves (see \ref{N:redmap1}).
We also say that an  abelian variety  $A$ over $K$ has \emph{semiabelian reduction} if the neutral component $\calN(A)^o$ of the N\'eron model of $A$ is a
semiabelian scheme. So Fact \ref{F:stabred} is saying that any abelian variety $A$ over $K$ has potentially semiabelian reduction, i.e.
there exists a finite field extension $K\subseteq K'$ such that $A':=A\times_K K'$ has semiabelian reduction.

Moreover, in the case where an abelian variety $A$ over $K$ has semiabelian reduction,
Grothendieck has shown (see \cite[Expos\'e IX, Thm. 10.4]{SGA7I}) that any polarization $\xi$ (for example a principal polarization) on $A$ gives rise to a monodromy pairing on the lattice of
characters $\Lambda(T):=\Hom(T,\Gm)$ of the maximal torus $T$ of the central fiber of $\calN(A)^o$, as in \eqref{E:can-ext}.

\begin{fact}[Grothendieck]\label{F:monpair}
Let $A$ be an abelian variety over $K$ and assume that $A$ has semiabelian reduction.
Denote by $\Lambda(T):=\Hom(T,\Gm)$ the lattice of characters of the biggest torus $T$ contained in the special fiber $\calN(A)^o_s$ of $\calN(A)^o$ as in
 \eqref{E:can-ext} and by $\Lambda(T)_{\bbR}:=\Lambda(T)\otimes_{\bbZ} \bbR$ the associated real vector space. Then any polarization $\xi$ on $A$ gives rise to a
 positive definite quadratic form
\begin{equation}\label{E:quadmon}
Q_{\xi}:\Lambda_{\bbR}\otimes \Lambda_{\bbR}  \to \bbR,
\end{equation}
which is moreover integral over $\Lambda$ (i.e. such that $Q_{\xi}(\Lambda,\Lambda)\subseteq \bbZ$).
\end{fact}

The quadratic form $Q_{\xi}$ in \eqref{E:quadmon} is called the \emph{monodromy pairing} associated to the polarized
abelian variety $(A,\xi)$.

We are now ready to  define the tropicalization $\trop: \AAg(K)\to \Agt$ appearing in the diagram  \eqref{E:fund-diag}.

\begin{lemmadefi}\label{D:tropmap-abvar}
The tropicalization map
$$\trop: \AAg(K)\to \Agt$$
is defined by sending $(A, \xi)\in \AAg(K)$ into the tropical p.p.   abelian variety $(\bbR^g/\Lambda', Q')\in \Agt$ such that:
\begin{itemize}
\item $\Lambda'$ is equal to $\Lambda':=\bbZ^{g-r}\oplus \Lambda(T')$, where $\Lambda(T')$ is the lattice of characters
of the torus $T'$ appearing in the extension \eqref{E:can-ext} with respect to some chosen field extension $K\subseteq K'$ such that
the neutral component $\calN(A')^o$ of the N\'eron model of $A':=A\times_K K'$ is a semiabelian scheme;
\item The quadratic form is identically zero on $\bbR^{g-r}:=\bbZ^{g-r}\otimes_{\bbZ}\bbR$ while on $\Lambda(T')_{\bbR}$  it is equal to
\begin{equation}\label{E:quad-mon}
Q'_{\Lambda(T')_{\bbR}\otimes \Lambda(T')_{\bbR}}:=\frac{Q_{\xi'}}{[K':K]},
\end{equation}
where $Q_{\xi'}$ is the monodromy pairing of Fact \ref{F:monpair} associated to the p.p. abelian variety $(A',\xi'):=(A,\xi)\otimes_K K'\in \AAg(K')$.
\end{itemize}
The so defined tropical p.p. abelian variety $(\bbR^g/\Lambda', Q') \in \Agt$  does not depend on the chosen  field extension $K\subseteq K'$ and is denoted by $\trop(A,\xi)$.
\end{lemmadefi}
\begin{proof}
Let $K'$ and $K''$ two finite field extensions of $K$ such that the neutral components $\calN(A')^o$ (resp. $\calN(A'')^o$) of $A':=A\times_K K'$ (resp. $A'':=A\times_K K''$)
are semiabelian schemes over the spectrum of the valuation ring $R'$ (resp. $R''$) of $K'$ (resp. $K''$).

As in the proof of Lemma-Definition \ref{D:red-curves}, we can find a finite field extension $K\subseteq L$, with valuation ring $S$,  that contains $K'$ and $K''$ as subfields.
It follows from \cite[Expos\'e IX, Cor. 3.3]{SGA7I} that $\calN(\wt{A})^o=\calN(A')^o\times_{\Spec R'} \Spec(S)=\calN(A'')^o\times_{\Spec R''}\Spec(S)$, where $\wt{A}:=A\times_K L$.
In particular we can canonically identify the lattice of characters $\Lambda(\wt{T})$ of the maximal torus $\wt{T}$ of $\calN(\wt{A})^o_s$ with the lattice of characters $\Lambda(T')$
(resp. $\Lambda(T')$) of the maximal torus $T'$ (resp. $T''$) of $\calN(A')^o_s$ (resp. $\calN(A'')^o_s$).

With respect to these canonical identifications, the monodromy pairing $Q_{\wt{\xi}}$ of the p.p. abelian variety $(\wt{A},\wt{\xi})=(A,\xi)\times_K L\in \AAg(L)$ is related to the
monodromy pairing $Q_{\xi'}$ (resp. $Q_{\xi''}$) of the p.p. abelian variety $(A',\xi'):=(A,\xi)\times_K K'\in \AAg(K')$ (resp. $(A'',\xi''):=(A,\xi)\times_K K''\in \AAg(K'')$)
via the formulas (see \cite[Expos\'e IX, (10.3.5)]{SGA7I}):
\begin{equation}\label{E:mon-basechange}
Q_{\wt{\xi}}=[L:K']Q_{\xi'}=[L:K'']Q_{\xi''} .
\end{equation}
By combining \eqref{E:quad-mon} and \eqref{E:mon-basechange}, we get that
$$Q'=\frac{Q_{\xi'}}{[K':K]}=\frac{Q_{\wt{\xi}}}{[L:K]}=\frac{Q_{\xi''}}{[K'':K]}=Q'',$$
where $Q'$ (resp. $Q''$) is the quadratic form associated to the extension $K\subseteq K'$ (resp. $K\subseteq K''$).
This shows that the definition of $\trop(X)$ is independent of the chosen field extension $K\subseteq K'$.
\end{proof}

\end{nota}

\subsection{The moduli stack $\AAgb$ of p.p. stable semi-abelic pairs}\label{S:ppSSAP}

The moduli stack $\AAg$ of principally polarized (p.p. for short) abelian varieties of dimension $g$ admits a modular compactification via p.p. stable
semi-abelic pairs.

\begin{defi}[Alexeev]
A \emph{p.p. stable semi-abelic pair} of dimension $g$ over $k$ is a triple
$( G \curvearrowright P, \Theta)$ where
\begin{itemize}
 \item[(i)] $G$ is a semiabelian variety of dimension $g$ over $k$, that is an algebraic
group which is an
extension of an abelian variety $A$ by a torus  $T$:
$$1\to T\to G \to A \to 0.$$
 \item[(ii)] $P$ is a seminormal, connected, projective variety of pure
dimension $g$.
\item[(iii)] $G$ acts on $P$ with finitely many orbits, and with connected
and reduced stabilizers contained in the toric part $T$ of $G$.
\item[(iv)] $\Theta$ is an effective ample Cartier divisor on $P$ which
does not contain
any $G$-orbit, and such that $h^0(P, \calO_P(\Theta))=1$.
\end{itemize}
\end{defi}

Recall that a $k$-variety $X$ is said to have seminormal singularities if any morphism $Y\to X$ from a $k$-variety $Y$ which is bijective on $k$-points is an isomorphism. 

\begin{remark}\label{R:abvar-abelic}
If $(A \curvearrowright P, \Theta)$ is a p.p. stable semi-abelic pair with $A$ being an abelian variety,
then $P$ is a $A$-torsor and the divisor $\Theta\subset P$ gives rise to a well-defined class $[\Theta]$ in the N\'eron-Severi group of $A$
which is a principal polarization on $A$. Conversely, every p.p. abelian variety $(A,\xi)$ can be obtained in  this way from a unique p.p. stable semi-abelic pair
$(A \curvearrowright P, \Theta)$. See \cite[Sec. 3]{Ale1} for more details on this correspondence.
\end{remark}

The following celebrated result is due to Alexeev \cite{Ale1}.

\begin{fact}[Alexeev]
The stack $\AAgmod$ of p.p. stable semi-abelic pairs of dimension $g$ is proper over $\Spec \bbZ$.
The stack $\AAg$ can be identified with  the open substack of $\AAgmod$ consisting of the p.p. stable semi-abelic pairs $( G \curvearrowright P, \Theta)$ such that
$G$ is an abelian variety.
\end{fact}

Unfortunately, the stack $\AAgmod$ is not irreducible (see \cite{Ale0}).  Therefore, only one of its irreducible components, called the \emph{main component} of $\AAgmod$
and denoted by  $\AAgb$,  will contain $\AAg$. Indeed, it is known that the normalization of the main component $\AAgb$  is isomorphic to the 2nd Voronoi
toroidal compactification  $\ov{\calA}^V_g$ of $\AAg$ (see \cite{AMRT} and \cite{NamT}).
To the best of our knowledge, it is not known whether the  main component $\AAgb$ is normal (see \cite{Bri}).

\vspace{0,2cm}

\begin{nota}{\emph{The stratification of $\AAgb(k)$}}
\label{N:strata-Ag}

According to general theory developed in \cite{Ale1}, to every p.p. stable semi-abelic pair $(G \curvearrowright P, \Theta)$ over $k$, it is naturally associated a $\bbZ^g$-period
integral paving of $\bbR^g$, up to the action of $\GL_g(\bbZ)$,  which captures the combinatorics of the $G$-orbits on $P$. Moreover, such a paving is a
 Delaunay decomposition if and if  $(G \curvearrowright P, \Theta)$ belongs to the main component $\AAgb$ (see \cite{Ale0}). In this way we get a stratification
 of $\AAgb$ into locally closed subsets parametrized by equivalence classes of Delaunay decompositions of $\bbR^g$.




\begin{fact}[Alexeev]\label{F:paving-Del}
\noindent
\begin{enumerate}[(i)]
\item \label{F:paving-Del1} A p.p. stable semi-abelic pair $(G \curvearrowright P, \Theta)\in \AAgmod(k)$ determines an equivalence class of a $\bbZ^g$-period
integral pavings of $\bbR^g$, which we denote by $[\Delta(G \curvearrowright P, \Theta)]$.

Furthermore, $(G \curvearrowright P, \Theta)$ belongs to the main
irreducible component $\AAgb(k)$ if and only if  $[\Delta(G \curvearrowright P, \Theta)]$ is an equivalence class of Delaunay decompositions of $\bbR^g$.

\item \label{F:paving-Del2}  The topological space $\AAgb(k)$ admits a stratification into disjoint locally closed subsets
$$\AAgb(k)=\coprod_{[\Delta]} \AAgb([\Delta]),$$
as $[\Delta]$ varies among all equivalence classes of Delaunay decompositions of $\bbR^g$ and
$$\AAgb([\Delta]):=\{(G \curvearrowright P, \Theta)\in \AAgb(k)\: : \: [\Delta(G \curvearrowright P, \Theta)]=[\Delta]\}.$$

Given two equivalence classes $[\Delta]$ and $[\Delta']$ of Delaunay decompositions of $\bbR^g$,
we have  that
$$\AAgb([\Delta])\subseteq \ov{\AAgb([\Delta'])}\Leftrightarrow [\Delta]\geq [\Delta'].$$
\end{enumerate}
\end{fact}
\begin{proof}
Part \eqref{F:paving-Del1} follows from the general structure theorems on p.p. stable semi-abelic varieties developed in \cite{Ale1} (see also \cite[Sec. 2]{Ale2} for a nice discussion).

Part \eqref{F:paving-Del2}: the strata of $\AAgb$ are the images of the strata of the 2nd Voronoi toroidal compactification $\ov{\calA}_g^V$
under the finite normalization map $\ov{\calA}_g^V \to \AAgb$ (see \cite{Ale0}) and the required properties are known for the strata of $\ov{\calA}_g^V$, as it follows from the general theory of
toroidal  compactifications of $\AAg$ (see \cite{AMRT} or \cite{NamT}). Therefore, the same properties hold for the strata of $\AAgb(k)$.

\end{proof}

\end{nota}

\begin{nota}{\emph{The reduction map $\red: \AAg(K)\to \AAgb(k)$}}
\label{N:redmap2}

We are now ready to define the reduction map $\red: \AAg(K)\to \AAgb(k)$ appearing in the diagram \eqref{E:fund-diag}.
Since the stack $\AAgb$ is proper, the valuative criterion of properness for stacks gives that for any map $f:\Spec K\to \AAg\subseteq \AAgb$ there exists
a finite extension $K'$ of $K$ with valuation ring $R'$ and a unique map $\phi:\Spec R'\to \AAgb$ such that the following diagram is commutative
$$
\xymatrix{
\Spec R'  \ar[rrrd]^{\phi}& & & \\
\Spec K' \ar[r] \ar[u] & \Spec K \ar[r]^f & \AAg \ar@{^{(}->}[r]& \AAgb.
}$$
In other words, given a p.p. abelian variety  $(A,\xi)\in \AAg(K)$, up to a finite extension $K\subseteq K'$ with valuation ring $R'$,
there exists a unique family of p.p. stable semi-abelic pairs $(\calG \curvearrowright \calP,\wt{\Theta})$ over $\Spec R'$,
called the \emph{stable semi-abelic reduction} of $(A,\xi)$ with respect to the extension $K\subseteq K'$, such that $(A,\xi)\times_K K'$ is the p.p. abelian variety
associated to the generic fiber of $(\calG \curvearrowright \calP,\wt{\Theta})$, according to Remark \ref{R:abvar-abelic}.
Note that the residue field of $R'$ is equal to $k$, since $k$ was assumed to be algebraically closed.

\begin{lemmadefi}\label{D:red-abvar}
The reduction map
$$\red: \AAg(K)\to \AAgb(k)$$
is defined by sending $(A,\xi)\in \AAg(K)$ to the central fiber $(\calG \curvearrowright \calP,\wt{\Theta})_s\in \AAgb(k)$ of the stable semi-abelic reduction
$(\calG \curvearrowright \calP,\wt{\Theta})$ of $(A,\xi)$ with respect to some finite field extension $K\subseteq K'$.
The isomorphism class of $(\calG \curvearrowright \calP,\wt{\Theta})_s\in \AAgb(k)$ does not depend on the chosen field extension $K\subset K'$ and is denoted by $\red(A,\xi)$.
\end{lemmadefi}
\begin{proof}
Same proof as in Lemma-Definition \ref{D:red-curves} based on the uniqueness of the stable semi-abelic reduction.
\end{proof}

Now that we have defined the reduction map $\red:\AAg(K)\to \AAgb(k)$ and the tropicalization map $\trop:\AAg(K)\to \Agt$, we can prove the second half of Theorem \ref{T:mainthm3}.

\begin{proof}[Proof of Theorem \ref{T:mainthm3}\eqref{T:mainthm3ii}]
We have to prove that for any p.p. abelian variety $(A,\xi)\in \AAg(K)$ it holds
$$[\Del_{\trop(A,\xi)}]=[\Delta(\red(A,\xi))],$$
following the notations of Definition \ref{D:Del-abvar} and of Fact \ref{F:paving-Del}. This is simply a restatement in our language of what Alexeev proved in \cite[Sec. 5.7]{Ale1}.

\end{proof}

\end{nota}

\section{The Torelli maps}\label{S:Tormaps}

 \subsection{The tropical Torelli map}\label{S:trop-Tor}

 The tropical Torelli map $\tgt:\Mgt\to \Agt$ has been constructed in \cite{BMV} and further studied in \cite{Cha}. In order to recall the definition of $\tgt$, we need first to recall the definition of
the  tropical Jacobian associated to a tropical curve.

\begin{defi}\label{D:Jac}
Let $C=(\Gamma,w,l)$ be a tropical curve of genus $g$ and total weight
$|w|$. The {\it Jacobian} $\Jac(C)$ of $C$ is the tropical p.p.  abelian variety
of dimension $g$ given by the real torus $(H_1(\Gamma,\R)\oplus \R^{|w|})/(H_1(\Gamma,\Z)\oplus\Z^{|w|})$ together with the positive semi-definite quadratic form $Q_{C}=Q_{(\Gamma,w,l)}$ which
vanishes identically on $\R^{|w|}$ and is given on
$H_1(\Gamma,\R)$ as
\begin{equation}\label{D:def-quad}
Q_{C}\left(\sum_{e\in E(\Gamma)}\alpha_e \cdot e\right)=\sum_{e\in E(\Gamma)} \alpha_e^2\cdot l(e).
\end{equation}
\end{defi}
In other words, the value of the quadratic form $Q_C$ on a cycle of $\Gamma$, seen as an element of $H_1(\Gamma,\bbR)$, is equal to its length measured with respect to
the length function $l$ of  the tropical curve $C$.

\begin{remark}\label{R:Mum-curves}
The referee noticed that the quadratic form $Q_C$ defined in \eqref{D:def-quad} appears already in the definition of the canonical polarization on the Jacobian of a Mumford curve, 
see \cite{Ger} and  \cite[Prop. 2.2]{vdP}.
\end{remark}

\begin{fact}[\cite{BMV}]\label{F:trop-Tor}
The map (called the \emph{tropical Torelli map}) Ê
$$\begin{aligned}
\tgt:  \Mgt& \longrightarrow  \Agt \\
C & \mapsto \Jac(C)
\end{aligned}
$$
is a continuous map.
\end{fact}

Indeed, it is proved in \cite[Thm. 5.1.5]{BMV} that $\tgt$ is a full map of stacky fans, i.e. that sends each strata of $\Mgt$ surjectively onto some strata of $\Agt$ via a linear map.
In order to make this result more precise, we need to recall the definition of the  Delaunay decomposition of $\bbR^g$ associated to a stable
weighted graph of genus $g$.

 \begin{defi}\label{D:graphs-Del}
Let $(\Gamma,w)$ be a stable weighted graph of genus $g$. Consider the positive semi-definite quadratic form  $Q_{(\Gamma,w)}$ on $H_1(\Gamma,\bbR)\oplus \bbR^{|w|}$
which is identically zero on $\bbR^{|w|}$ and is given on $H_1(\Gamma,\bbR)$ by
\begin{equation}
Q_{(\Gamma,w)}\left(\sum_{e\in E(\Gamma)}\alpha_e \cdot e\right)=\sum_{e\in E(\Gamma)} \alpha_e^2.
\end{equation}
By fixing an isomorphism of free abelian groups $\phi:H_1(\Gamma,\bbZ)\stackrel{\cong}{\to} \bbZ^{b_1(\Gamma)}$, we can view $Q_{(\Gamma,w)}$ as  a positive semi-definite quadratic form
on $\bbR^g$. The equivalence class $[\Del_{Q_{(\Gamma,w)}}]$ of the induced Delaunay decomposition of $\bbR^g$ (which clearly does not depend upon the chosen isomorphism $\phi$)
is called the Delaunay decomposition of $(\Gamma,w)$ and is denoted by $[\Del(\Gamma,w)]$.
\end{defi}

 \begin{remark}
It is well known  that an equivalent definition  of $[\Del(\Gamma,w)]$ is the following.
Each edge $e$ of $\Gamma$ gives rise to a linear functional $e^*$ on $H_1(\Gamma,\bbR)\oplus \bbR^{|w|}$ which is identically zero on $\bbR^{|w|}$ and
it is equal on $H_1(\Gamma,\bbR)$ to
$$e^*\left(\sum_{f\in E(\Gamma)}\alpha_f \cdot f\right)=\alpha_e.$$
After fixing an isomorphism $\phi:H_1(\Gamma,\bbZ)\stackrel{\cong}{\to} \bbZ^{b_1(\Gamma)}$ as before, the Delaunay decomposition $[\Del(\Gamma,w)]$ is the
$\bbZ^g$-periodic integral paving of $\bbR^g$ consisting of all polyhedra which are cut out
by all hyperplanes of equation $e^*=n$ for $e\in E(\Gamma)$ and $n\in \Z$. We refer the reader to \cite[Sec. 3.2]{CV1} for more details on the Delaunay decompositions
associated to graphs.
 \end{remark}

\begin{fact}[\cite{BMV}]\label{F:tropTor-strata}
The tropical Torelli map $\tgt$ sends the strata $\Mgt(\Gamma,w)\subset \Mgt$ surjectively onto the strata $\Agt([\Del(\Gamma,w)]\subset \Agt$, i.e.
$$\tgt(\Mgt(\Gamma,w))=\Agt([\Del(\Gamma,w)]$$
for each stable weighted graph $(\Gamma,w)$ of genus $g$.
\end{fact}

We can now prove  the second half of Theorem \ref{T:mainthm1}.

\begin{thm}\label{T:commuta2}
The following diagram is commutative
\begin{equation}\label{E:comdiag2}
\xymatrix{
  \MMg(K) \ar[d]^{t_g}\ar[r]^{\rm trop} & \Mgt \ar[d]^{t_g^{\rm tr}}\\
 \AAg(K)  \ar[r]^{\rm trop} & \Agt\\
}
\end{equation}
\end{thm}
\begin{proof}
Let $X$ be an element of $\MMg(K)$, i.e. a  connected smooth projective curve of genus $g$ over $K$.

Assume first that $X$ has a stable model over $\Spec R$, i.e. there exists a family $\calX\to \Spec R$ of stable curves of genus $g$ such that the generic fiber $\calX_{\eta}$
is isomorphic to $X$.

According to Lemma-Definition \ref{D:tropmap-curves}, $\trop(X)$ has combinatorial type equal to the dual weighted graph $(\Gamma_{\calX_s}, w_{\calX_s})$
of the special fiber $\calX_s$ of $\calX$ and its length function $l:E(\Gamma_{\calX_s})\to \bbR_{>0}$ is such that, for every $e\in E(\Gamma_{\calX_s})$:
\begin{equation}\label{E:def-leng}
l(e)=w_{n_e}
\end{equation}
where $w_{n_e}$ is the width of the node $n_e\in \calX_s$ corresponding to $e$  (see \ref{N:tropmap1}).

By blowing up each node $n$ of the central fiber $\calX_s$ a number of times equal to $(w_n-1)$, we get that
a new family of nodal curves $\calY\to \Spec R$ such that $\calY_{\eta}\cong X$ and $\calY$ is regular. The central fiber $\calY_s$ of $\calY$ is a nodal (non stable, in general)
curve which is obtained from $\calX_s$ by inserting at each node $n$ of $\calX_s$ a  chain of smooth rational curves of length equal to $(w_n-1)$.
This implies that the dual graph $\Gamma_{\calY_s}$ of $\calY_s$ is obtained from the dual graph $\Gamma_{\calX_s}$ of $\calX_s$ by subdividing each edge
$e\in E(\Gamma_{\calX_s})$ a number  of times equal to $(w_{n_e}-1)$.
In particular, we have a canonical isomorphism
$H_1(\Gamma_{\calX_s},\bbZ)\cong H_1(\Gamma_{\calY_s},\bbZ)$. Moreover, the pull-back map induces a canonical
isomorphism $J(\calX_s)\stackrel{\cong}{\to} J(\calY_s)$ between the generalized Jacobians of $\calX_s$ and of $\calY_s$.

According to \cite[Sec. 9.3, Thm. 7]{BLR}, there exists a scheme $J(\calY)$ smooth and separated over $ \Spec R$, called the \emph{relative Jacobian} of the family $\calY\to \Spec R$,
such that its generic fiber  $J(\calY)_{\eta}$ is isomorphic to the Jacobian $J(\calY_{\eta})=J(X)$ of the generic fiber and its special fiber $J(\calY)_s$ is isomorphic to
the generalized Jacobian $J(\calY_s)$ of the special fiber. Moreover, a well-known result
of Raynaud (see \cite[Sec. 9.5, Thm. 4]{BLR}) says that, since $\calY$ is regular, the relative Jacobian $J(\calY)$ of $\calY\to \Spec R$ is isomorphic to the neutral component
$\calN(J(X))^o$ of the N\'eron model of the Jacobian $J(X)$ of $X\cong \calY_{\eta}$. In particular, since $J(\calY)$ is a semiabelian scheme over $\Spec R$, the Jacobian $J(X)$ of $X$
has semiabelian reduction over $K$. Note that the lattice of characters of the maximal subtorus of $J(\calY)_s$ is canonically isomorphic to
$H_1(\Gamma_{\calY_s},\bbZ)\cong H_1(\Gamma_{\calX_s},\bbZ)$.

Now the Picard-Lefschetz formula (see \cite[Expos\'e IX, Thm. 12.5]{SGA7I}) says that the monodromy pairing $Q_{\xi}$ on $H_1(\Gamma_{\calY_s},\bbR)$ associated to the
principal polarization $[\TXg]$ on $J(X)$ induced by the theta divisor $\Theta_X\subset \Pic^{g-1}(X)\cong J(X)$ (see Fact \ref{F:monpair}) is equal to
\begin{equation}\label{E:quad-Y}
Q_{[\TXg]}\left(\sum_{e\in E(\Gamma_{\calY_s})}\alpha_e \cdot e\right)=\sum_{e\in E(\Gamma_{\calY_s})}  \alpha_e^2
\end{equation}
Using the canonical isomorphism $H_1(\calY_s,\bbZ)\cong H_1(\calX_s,\bbZ)$, it is immediate to check that the above monodromy pairing $Q_{\xi}$ on  $H_1(\calY_s,\bbZ)$ becomes
isomorphic to the quadratic form $Q_{\trop(X)}$ on $H_1(\Gamma_{\calX_s},\bbR)$ defined by \eqref{D:def-quad}. By comparing Lemma-Definition \ref{D:tropmap-abvar} with Definition
\ref{D:Jac}, we see that $\trop(J(X),[\TXg])=\Jac(\trop(X))$, which shows the commutativity of the diagram \eqref{E:comdiag2}.

In the general case (when $X$ does not have a stable reduction over $K$), we can find a finite field extension
$K\subseteq K'$ with valuation ring $R'$ such that  the base change $X_{K'}$ of $X$ to $K'$ admits a stable reduction $\calX'\to \Spec R'$.
We can repeat the above argument working with the family $\calX'\to \Spec R'$ with the following two modifications: in defining the length  of the tropical curve
$\trop(X)$ we have to divide the right hand side of \eqref{E:def-leng} by $[K':K]$ and in defining the quadratic form giving $\trop(X,[\Theta_X])$ we have to divide the
monodromy pairing \eqref{E:quad-Y} by $[K':K]$. Clearly, with these two modifications, the equality $\trop(J(X),[\TXg])=\Jac(\trop(X))$ continues to hold,  and
the  commutativity of the diagram \eqref{E:comdiag2} in the general case follows.

\end{proof}

\begin{nota}{\emph{The fibers of the tropical Torelli map $\tgt$}}\label{S:fibers-tgt}

The aim of this subsection is to recall the description of the fibers of $\tgt$ obtained in \cite{CV1}.

A first step is to describe the strata $\Mgt(\Gamma,w)$ of $\Mgt$ that are mapped to the same stratum $\Agt([\Delta])$ of $\Agt$.
To this aim, we recall the following classical definition, due to Whitney.

\begin{defi}[Whitney]\label{D:cyc-iso}
Two graphs $\Gamma_1$ and $\Gamma_2$ are said to be \emph{cyclically equivalent} (or $2$-isomorphic) ,
and we write $\Gamma_1\equiv_{\cyc} \Gamma_2$,
if there exists a bijection $\phi:E(\Gamma_1)\to E(\Gamma_2)$
inducing a bijection between cycles of $\Gamma_1$ and cycles of $\Gamma_2$.
We denote by $[\Gamma]_{\cyc}$ the cyclic isomorphism class
of a graph $\Gamma$.
\end{defi}


In the sequel, graphs with  edge-connectivity at least $3$ will play an important role. Here it is the standard definition.

\begin{defi}\label{D:3conn}
Let $\Gamma$ be a connected graph.
\begin{enumerate}[(i)]
\item An edge $e$ of $\Gamma$ is called a \emph{separating edge} (or a coloop or a bridge) if the graph obtained by removing $e$ is disconnected.
Two edges $e$ and $f$ are said to be \emph{coparallel}Ê if neither of them is a separating edge and the graph obtained by removing $e$ and $f$ is disconnected.

 \item
 $\Gamma$ is said to be \emph{$3$-edge-connected} if $\Gamma$ does not have separating edges nor pairs of coparallel edges.
 \end{enumerate}
 \end{defi}
It is easy to see that the property of being coparallel defines an equivalence relation on the set of non separating edges of $\Gamma$.
The equivalence classes with respect to this equivalence relation are called \emph{coparallel classes}
\footnote{These equivalence classes were called C1-sets in \cite{CV1}, with a terminology coming from algebraic geometry (see \cite{CV2}). Here we choose to use
the more graph-theoretic terminology of coparallel in order to suggest that the coparallel equivalence relation is the dual notion (in the sense of matroid theory)
of the parallel equivalence relation.}.

There is a canonical way of obtaining
a $3$-edge-connected graph, up to cyclic isomorphism, starting from any graph.

 \begin{defi}\label{D:connecti}
 Let $\Gamma$ be a connected graph.

A {\it 3-edge-connectivization} of $\Gamma$ is a graph, denoted by $\Gamma^{3}$,
obtained from $\Gamma$ by contracting all the separating edges
and all but one among the edges of each coparallel class of $\Gamma$.
The cyclic isomorphism class of $\Gamma^3$ (which is well-defined and it does not depend on the choice of $\Gamma^3$)
is called the $3$-edge-connectivization class of $\Gamma$ and is denoted by $[\Gamma^3]_{\cyc}$.
\end{defi}

After these preliminary definitions, we can now recall the following result (proved in \cite[Sec. 3.2]{CV1}) which characterize the stable weighted graphs that
have the same associated Delaunay decomposition.

\begin{fact}[Caporaso-Viviani]\label{F:same-Del}
Let $(\Gamma_1,w_1)$ and $(\Gamma_2,w_2)$ two stable weighted graphs of genus $g$. Then
$$[\Del(\Gamma_1,w_1)]=[\Del(\Gamma_2,w_2)]\Leftrightarrow [\Gamma_1^3]_{\cyc}=[\Gamma_2^3]_{\cyc}.$$
\end{fact}

\vspace{0,2cm}

We turn now to the following natural question:  for which tropical curves $C, C'\in \Mgt$ it holds that $\tgt(C)=\tgt(C')$?
We first need a couple of definitions.


\begin{defi}\label{D:cyc-isomo-trop}
Two tropical curves $C=(\Gamma, w, l)$ and $C'=(\Gamma',w',l')$
are cyclic isomorphic, and we write $C\equiv_{\cyc} C'$, if there exists a bijection
$\phi:E(\Gamma)\to E(\Gamma')$, commuting with the length functions $l$ and
$l'$, that induces a cyclic isomorphism between $\Gamma$ and $\Gamma'$.
We denote by $[C]_{\cyc}$ the cyclic isomorphism equivalence class of a tropical
curve $C$.
\end{defi}

Similarly to definition \ref{D:3conn}, we have the following

\begin{lemmadefi}\label{3-conn-trop}
Let $C=(\Gamma,l,w)$ a tropical curve. A $3$-edge-connecti\-vization
of $C$ is a tropical curve $C^3=(\Gamma^3,l^3,w^3)$ obtained
in the following manner:
\begin{enumerate}[(i)]
\item $\Gamma^3$ is a $3$-edge-connectivization of $\Gamma$
in the sense of Definition \ref{D:connecti}, i.e. $\Gamma^3$ is obtained
from $\Gamma$ by contracting all the separating edges of $\Gamma$
and, for each coparallel class $S$ of $\Gamma$, all but one the edges of $S$,
which we denote by $e_S$;
\item $w^3$ is the weight function on $\Gamma^3$ induced by the weight
function $w$ on $\Gamma$ in the following way: at each contraction of some edge of $\Gamma$, the new vertex has weight with respect to $w^3$ equal to
the sum of the weights with respect to $w$ of the two vertices mapping to it;
\item $l^3$ is the length function on $\Gamma^3$ given
by $$l^3(e_S)=\sum_{e\in S} l(e),$$
for each coparallel class $S$ of $\Gamma$.
\end{enumerate}
The cyclic isomorphism class of $C^3$ is well-defined;
it will be called the $3$-edge-connectivization class of $C$
and denoted by $[C^3]_{\cyc}$.
\end{lemmadefi}

The following result was proved by Caporaso-Viviani in \cite[Thm. 4.1.9]{CV1} in the case when the total weights of the tropical curves are zero
and then the proof was easily adapted to the general case by Brannetti-Melo-Viviani in \cite[Thm. 5.3.3]{BMV}.

\begin{fact}[Caporaso-Viviani]\label{F:tropTorfibers}
Let $C_1, C_2\in \Mgt$. Then
$$\tgt(C_1)=\tgt(C_2)\Leftrightarrow [C_1^3]_{\cyc}=[C_2^3]_{\cyc}.$$
\end{fact}

The previous Fact allows us to describe a locally closed subset of $\Mgt$ where the tropical Torelli map $\tgt$ is injective.

Recall that a connected graph $\Gamma$ is said to be \emph{$3$-vertex-connected} if, for any pair $\{v_1,v_2\}$ of
(possibly equal) vertices of $\Gamma$,  the graph $\Gamma\setminus \{v_1,v_2\}$ obtained from $\Gamma$ by removing
$v_1$, $v_2$ together with all the edges that are adjacent to them is connected. It is easily seen that a
$3$-vertex-connected graph is also $3$-edge-connected in the sense of Definition \ref{D:3conn}.

\begin{cor}\label{C:inj-tgt}
The tropical Torelli map $\tgt:\Mgt\to \Agt$ is injective on the locally closed subset $F$ of $\Mgt$ consisting of tropical curves
$C$ whose combinatorial type $(\Gamma,w)$  is such that $\Gamma$ is $3$-vertex-connected and $g(\Gamma)=g$.
\end{cor}
\begin{proof}
Since $F$ is the union of strata of $\Mgt$, it is clear that $F$ is locally closed.
Now it follows from a classical result of Whitney (see \cite[Thm. 2.2.4]{CV1} and the references therein) that if
$C_1, C_2\in F$ then
$$[C_1^ 3]_{\rm cyc}=[C_1]_{\rm cyc}=[C_2]_{\rm cyc}=[C_2^ 3]_{\rm cyc}\Leftrightarrow C_1=C_2,$$
which, together with Fact \ref{F:tropTorfibers}, finishes the proof.
\end{proof}

\end{nota}

\subsection{The compactified Torelli morphism}\label{S:comp-Tor}

The Torelli morphism $t_g:\MMg\to \AAg$ can be extended to a modular morphism $\tgb:\MMgb\to \AAgb$,  as shown by Alexeev in \cite{Ale2}. Before recalling his result,
we need the following definitions.

\begin{defi}
Let $X$ be a stable curve of arithmetic genus $g$ over $k$.
\begin{enumerate}[(i)]
\item The \emph{generalized Jacobian}Ê $J(X)$ of $X$ is the semiabelian variety parametrizing line bundles on $X$ of multidegree $0$, i.e. having degree $0$ on each
irreducible component of $X$.
\item The \emph{degree $g-1$ canonical compactified Jacobian} $\PXgb$ of $X$ is the moduli space of torsion-free, multirank $1$ (i.e. having rank $1$ on each
irreducible component of $X$) and degree $g-1$ sheaves $\calI$ on $X$ that are $\omega_X$-semistable.
\item  The \emph{theta divisor} $\TXg$ of $X$ is the closed reduced subscheme of $\PXgb$ defined by
$$\TXg:=\{\calI\in \PXgb \: : \: h^0(X, \calI)>0\}.$$
\end{enumerate}
\end{defi}

It is well known  that $J(X)$ is the extension of the Jacobian $J(\wt{X})$ of the normalization $\wt{X}$ of $X$ (which has dimension equal to the geometric genus of $X$) by a torus
whose lattice of characters can be naturally identified with the first homology group $H_1(\Gamma_X,\bbZ)$ of the dual graph $\Gamma_X$ of $X$.
 Note also that the generalized Jacobian $J(X)$ acts naturally on $\PXgb$ by tensor product.

The following result was proved by Alexeev in \cite{Ale2}.

\begin{fact}[Alexeev]\label{F:comp-tg}
\noindent
\begin{enumerate}[(i)]
\item \label{F:comp-tg1} For any stable curve $X$ of genus $g$, the triple $(J(X)\curvearrowright \PXgb, \TXg)$ is a p.p. stable semi-abelic pair of dimension $g$.
\item \label{F:comp-tg2} The Torelli morphism $t_g:\MMg\to \AAg$ extends to a morphism $\tgb:\MMgb\to \AAgb$, called the \emph{compactified Torelli morphism},
which sends a stable curve $X$ into the p.p. stable semi-abelic pair   $(J(X)\curvearrowright \PXgb, \TXg)$.
\item \label{F:comp-tg3} The compactified Torelli morphism $\tgb$ sends the stratum $\MMgb(\Gamma,w)\subset \MMgb(k)$ into the stratum
$\AAgb([\Del(\Gamma,w)])\subseteq \AAgb(k)$.
\end{enumerate}
\end{fact}

Using the above result, we can now easily prove the first half of Theorem \ref{T:mainthm1}.

\begin{thm}\label{T:commuta1}
The following diagram is commutative
\begin{equation}\label{E:comdiag1}
\xymatrix{
\MMgb(k)\ar[d]^{\tgb} &  \MMg(K) \ar[d]^{t_g} \ar[l]_{\rm red} \\
\AAgb(k)  & \AAg(K)  \ar[l]_{\rm red} \\
}\end{equation}
\end{thm}
\begin{proof}
Consider an element of $\MMg(K)$, i.e. a morphism $f:\Spec K\to \MMg\subset \MMgb$. By applying the valuative criterion of properness  to the stack
$\MMgb$, we get that, up to a finite extension
$K\subseteq K'$ with valuation ring $R'$, we can extend the morphism $f$ to a morphism $\phi:\Spec R'\to \MMgb$. In this way we get a  commutative diagram
\begin{equation*}
\xymatrix{
& & & \Spec R' \ar[llld]_{\phi} &\Spec k \ar@{_{(}->}[l]_s\\
\MMgb\ar[d]^{\tgb} &  \MMg \ar[d]^{t_g} \ar@{_{(}->}[l] & \Spec K \ar[l]_f & \Spec K'\ar[l] \ar@{^{(}->}[u]^{\eta} &  \\
\AAgb  & \AAg  \ar@{_{(}->}[l] & &  & \\
}\end{equation*}
where the upper triangle is commutative by construction and the bottom left square is commutative by Fact \ref{F:comp-tg}.
As explained in \ref{N:redmap1}, the first reduction map $\red:\MMg(K)\to \MMgb(k)$ sends
the morphism $f:\Spec K\to \MMg$ into the morphism $\phi\circ s:\Spec k\to \MMgb$.   Analogously, as explained in \ref{N:redmap2}, the second reduction map
$\red:\AAg(K)\to \AAgb(k)$ sends the morphism  $t_g\circ f:\Spec K\to \AAg$ into the morphism $\tgb\circ \phi \circ s:\Spec k\to \AAgb$.
The commutativity of the diagram \eqref{E:comdiag1} now follows.
\end{proof}

 \begin{nota}{\emph{The fibers of the compactified Torelli morphism $\tgb$}}\label{S:fibers-tgb}

The aim of this subsection is to recall the description of the fibers of the compactified Torelli morphism
$\tgb:\MMgb(k)\to \AAgb(k) $ obtained in \cite{CV2}.
Before doing that, we need to recall some definitions.

\begin{defi}\label{D:sep-blocks}
Given a stable curve $X$ of genus $g$, consider its partial normalization $\tau:\wh{X}\to X$ at the separating nodes of $X$, i.e. the nodes $n$ of $X$ such that the partial normalization
of $X$ at $n$ is disconnected.
Write
$$\wh{X}:=X_1\coprod \cdots \coprod X_s, $$
where $s\in \bbN$ and each $X_i$ is a connected nodal curve. We call the curves $\{X_1,\cdots, X_s\}$ the \emph{separating blocks}Ê of $X$.
\end{defi}

Note that the separating blocks of $X$ are connected nodal curves free from separating nodes, which are however not stable in general.
Consider one of the separating blocks $X_i$. If $X_i$ has arithmetic genus $p_a(X_i)$ equal to zero, then $X_i\cong \bbP^1$.
Otherwise, $X_i$  is \emph{semistable}, i.e. it is a  connected nodal curves such that its canonical line $\omega_{X_i}$
has non-negative (possibly zero) degree on each irreducible component of $X_i$. If, moreover, $p_a(X_i)\geq 2$, then
 we can consider its \emph{stabilization}, denoted by
$\ov{X}_i$, which is the image of $X_i$ under the map given   by $|\omega_{X_i}^m|$ for $m$ sufficiently large (indeed any $m\geq 3$ suffices).
It is easy to see that $\ov{X}_i$ is obtained from $X_i$ by contracting to a node  all the exceptional subcurves $E\subset X_i$, i.e. subcurves $E\cong \bbP^1$ such that
$E$ intersect the complementary subcurve $\ov{X_i\setminus E}$ in two points.  We can extend the definition of the stabilization $\ov{X}_i$ to the case where
$p_a(X_i)=1$ as it follows: if $X_i$ is smooth then we set $\ov{X}_i=X_i$; if $X_i$ is not smooth (which happens exactly when
$X_i$ is a cycle of rational smooth curves)  then we set $\ov{X}_i$ be equal to the rational irreducible curve with one node.


For a nodal curve $X$ without separating nodes, we can partition the set $X_{\rm sing}$ of nodes of $X$ into C1-sets
as in  \cite[Lemma-Definition 2.1.1]{CV2}.

\begin{defi}\label{D:C1-part}
Let $X$ be a connected nodal curve free from separating nodes. A \emph{separating pair} $\{n_1,n_2\}$ of nodes of $X$ is a pair consisting of two nodes $n_1$ and $n_2$ of $X$
such that the partial normalization of $X$ at $n_1$ and $n_2$ is disconnected.

Being a separating pair of nodes is an equivalence relation on the set of nodes $X_{\rm sing}$ of $X$ and we call the associated equivalence classes the
 \emph{C1-sets} of $X$. We denote by $\Set X$ the collection of all C1-sets of $X$.

\end{defi}
Note that the C1-sets of $X$ correspond exactly to the coparallel classes (see Definition \ref{D:3conn}) of edges in the dual graph $\Gamma_X$ of $X$.

We now recall the definition of C1-equivalence introduced in \cite[Def. 2.1.5]{CV2}.

\begin{defi}[C1-equivalence]
\label{D:C1}
Let $X$ and $X'$ be connected nodal curves free from separating nodes;
 denote by $\nu:X^{\nu} \to X$ and $\nu': X^{'\nu}\to X'$  their normalizations.
 $X$ and $X'$ are {\it {C1-equivalent}} if the following conditions hold
\begin{enumerate}[(A)]
\item\label{C1n} There exists an isomorphism $\phi: X^{\nu}\stackrel{\cong}{\to} X'^{\nu}$.
\item \label{C1S}
There exists a bijection between their C1-sets
$$
\psi:\Set X \stackrel{\cong}{\to} \Set X'
$$
such that $\phi(\nu^{-1}(S))=\nu'^{-1}(\psi(S))$.
\end{enumerate}
\end{defi}

The above C1-equivalence relation can be realized via a sequence of  twisting operations at pairs of separating nodes.

\begin{defi}[Twist-equivalence]\label{D:twists}
Let $X$ be a connected nodal curve free from separating nodes. Consider a separating pair $\{n_1,n_2\}$ of nodes of $X$. Denote by $\epsilon:Y\to X$ the partial normalization of $X$ at
$n_1$ and $n_2$, denote by $Y_1$ and $Y_2$ the connected components of $Y$  and let $\nu^{-1}(n_1)=\{p^1,p^2\}$
and $\nu^{-1}(n_2)=\{q^1, q^2 \}$  with $p^1, q^1\in Y_1$ and $p^2, q^2\in Y_2$. In particular $X$ is obtained from $Y$ by gluing $p_1$ with $p_2$ and $q_1$ with $q_2$;
or in symbols
$$X=\frac{Y}{\{p^1\sim p^2, q^1\sim q^2\}}.$$
The \emph{twist} of $X$ at $\{n_1,n_2\}$ is the curve $X'$ obtained from $Y$ by gluing $p^1$ with $q^2$ and $q^1$ with $p^2$, or in symbols:
$$X'=\frac{Y}{\{p^1\sim q^2, q^1\sim p^2\}}.$$

We say that two connected nodal curves free from separating nodes $X$ and $X'$ are \emph{twist-equivalent}Ê if $X'$ can be obtained from $X$ via a sequence of twisting at
separating pairs of  nodes.
\end{defi}

\begin{lemma}\label{L:C1-twist}
Let $X$ and $X'$ be connected nodal curves free from separating nodes. Then $X$ and $X'$ are C1-equivalent if and only if they are twist-equivalent.
\end{lemma}
\begin{proof}
This follows from the discussion in \cite[Sec. 2.3.2]{CV2}.
\end{proof}

With this definitions, we can now recall the description of the fibers of $\tgb$ obtained in \cite{CV2}.

\begin{fact}[Caporaso-Viviani]\label{F:Tor-stable}
Let $X, X'\in \MMgb(k)$ two stable curves of genus $g$.
Denote by $\{X_1, \cdots, X_r\}$ (resp. $\{X'_1,\cdots, X'_{r'}\}$)Ê the separating blocks of $X$ (resp. $X'$) that have arithmetic genus greater than zero.

The following are equivalent:
\begin{enumerate}[(i)]
\item \label{T:Tor-stable1} $\tgb(X)=\tgb(X')$.
\item \label{T:Tor-stable2} We have that $r=r'$ and, up to reordering  the separating blocks, we have that $\ov{X}_i$ is C1-equivalent to $\ov{X'}_i$ for each $1\leq i\leq r=r'$.
\item \label{T:Tor-stable3} We have that $r=r'$ and, up to reordering  the separating blocks, we have that $\ov{X}_i$ is twist-equivalent to $\ov{X'}_i$ for each $1\leq i\leq r=r'$.
\end{enumerate}
\end{fact}
\begin{proof}
The equivalence $\eqref{T:Tor-stable1}\Leftrightarrow \eqref{T:Tor-stable2}$ is a restatement of  \cite[Thm. 2.1.7]{CV2}.
The equivalence $\eqref{T:Tor-stable2}\Leftrightarrow \eqref{T:Tor-stable3}$ follows from Lemma \ref{L:C1-twist}.
\end{proof}
\end{nota}


\begin{cor}\label{C:inj-tgb}
The compactified Torelli morphism $\tgb:\MMgb(k)\to \AAgb(k)$ is injective on the open subset of $\MMgb(k)$ consisting
of stable curves without separating nodes nor separating pairs of nodes.
\end{cor}

We end this subsection with a reformulation of Fact \ref{F:Tor-stable} in the case of curves $X$ free from separating nodes and not hyperelliptic (in the sense of \cite[Def. 3.9]{Cat}), i.e. such that there does not exist two smooth points $p, q\in X$ with $h^0(X,\calO_X(p+q))=2$.
Note that if $X$ is a stable curve of genus $g\geq 2$ free from separating nodes, then $\omega_X$ is base point free by \cite[Thm. D]{Cat} and hence the complete linear
system $|\omega_X|$ defines a morphism $\phi_{|\omega_X|}:X\to \bbP^{g-1}$ (well-defined only up to composing with a projectivity of $\bbP^{g-1}$),
called the \emph{canonical morphism}.
Using results of Catanese \cite{Cat} and Catanese-Franciosi-Hulek-Reid \cite{CFHR},
the image of the canonical morphism  is described as it follows.

\begin{thm}\label{F:can-im}
Let $X$ be a stable curve $X$ of genus $g\geq 2$ free from separating nodes and not hyperelliptic.
Then
$\phi_{|\omega_X|}(X)$ is the curve obtained from $X$ by identifying all the nodes belonging to the same
C1-set $S$ into a unique point which is moreover a seminormal singularity of multiplicity $2|S|$, i.e. analytically isomorphic to the origin in the union of the coordinates axes of $\bbA^n$
with $n=2|S|$.
\end{thm}
\begin{proof}
The canonical morphism $\phi_{|\omega_X|}$ is an isomorphism away from the separating nodes of $X$, as it follows from the proof of \cite[Thm. 3.6]{CFHR}.
Moreover from \cite[Thm. E]{Cat} it follows that $\phi_{|\omega_X|}$ sends all the nodes belonging to a C1-set $S$ into the same point $p_S$ and, moreover, that
$p_S\neq p_{S'}$ if $S$ and $S'$ are two distinct C1-sets. Finally, from \cite[Rmk. 3.8]{Cat}, it follows that each point $p_S$ is analytically isomorphic to the origin in the union of
the coordinates axes of $\bbA^n$  with $n=2|S|$.
\end{proof}

 \begin{thm}\label{T:Tor-canonical}
Let $X, X'\in \MMgb(k)$ two stable curves of genus $g$ free from separating nodes and not hyperelliptic. Then
$$\tgb(X)=\tgb(X')\Leftrightarrow \phi_{|\omega_X|}(X)\cong \phi_{|\omega_{X'}|}(X').$$
 \end{thm}
\begin{proof}
According to the above Fact \ref{F:can-im}, the curve $\phi_{|\omega_X|}(X)$ can be constructed from the normalization $X^{\nu}$ of $X$ by gluing together
the points $\nu^{-1}(S)$, for each C1-set $S$, into a seminormal singular point (note there is a unique way of performing this gluing, i.e. seminormal curve singularities
do not have local moduli).  Therefore, $\phi_{|\omega_X|}(X)$ depends only on the C1-equivalence class of $X$ and, conversely, we can recover $X$ up to C1-equivalence
from the curve $\phi_{|\omega_X|}(X)$. The Theorem now follows from Fact \ref{F:Tor-stable}.
\end{proof}

\section{The anticontinuity of the reduction maps}\label{S:anti-redmap}

The aim of this Section is to prove the anticontinuity of the reduction maps appearing in the diagram \eqref{E:fund-diag}.
Indeed, this will follow from Corollary \ref{C:red-proper} which says that the same is true for any proper stack.

Recall that, as usual (see \ref{S:notations}),  we fix a complete DVR $R$  with maximal ideal $\m$ and we assume that
its residue field $k:=R/\m$ is algebraically closed. 
Given an element $x\in R$, we denote by $\ov{x}\in k$ its reduction modulo the maximal ideal $\m$.

The following well-known lemma is the key result for what follows.

\begin{lemma}\label{L:redAn}
For any positive integer $n$ consider the reduction map
$$
\begin{aligned}
\red: R^n& \longrightarrow k^n \\
\un x=(x_1,\cdots, x_n) & \mapsto \red(\un x):=(\ov{x}_1,\cdots, \ov{x}_n).
\end{aligned}
$$
If we put the non-archimedean topology on $R^n$ and the Zariski topology on $k^n$, then the reduction map $\red$ is \emph{anticontinuous}, i.e. the inverse image of a closed subset
is an open subset or, equivalently, the inverse image of an open subset is a closed subset.
\end{lemma}
\begin{proof}
Consider a Zariski closed subset $C\subseteq k^n$. By definition of the Zariski topology on $k^n$, this means that there exists a finite number of polynomials
$F_1,\cdots, F_r\in k[t_1,\cdots, t_n]$ such that
\begin{equation}\label{E:Zar-closed}
C=\bigcap_{i=1}^r V(F_i):=\bigcap_{i=1}^r \{\un z\in k^n\: : \: F_i(\un z)=0\}.
\end{equation}
For any $1\leq i\leq r$, we choose a polynomial $\wt{F}_i\in R[t_1,\cdots, t_n]$ whose reduction $\red(\wt{F}_i)$ is equal to $F_i$, where the reduction of a polynomial
with coefficients in $R$ is the polynomial with coefficients in $k$ obtained by reducing modulo $\m$ each of its coefficients.
For an element $\un x\in R^n$, we see that
\begin{equation}\label{E:red0}
0=F_i(\red(\un x))=\red(\wt{F}_i(\un x))\Leftrightarrow |\wt{F}_i(\un x)|<1.
\end{equation}
For any polynomial $F\in R[t_1,\cdots, t_n]$, consider the evaluation function
$$\begin{aligned}
\Phi_{F}: R^n & \longrightarrow R \\
\un x & \mapsto \Phi_F(\un x):=F(\un x).
\end{aligned}
$$
Clearly, the function $\Phi_F$ is continuous  with respect to the non-archimedean topology on the domain and the codomain.
Using the evaluation functions, the equivalence \eqref{E:red0} can be rewritten as
\begin{equation}\label{E:inv-im}
\red^{-1}(V(F_i))=\Phi_{\wt{F}_i}^{-1}(\m).
\end{equation}
Therefore the inverse image of each $V(F_i)$ under the reduction map $\red$ is open (recall that $\m\subset R$ is open in the non-archimedean topology being equal to the open
ball centered at $0$ and of radius $1$, see \ref{S:notations});
the same is true for $\red^{-1}(C)$ because of the representation as in \eqref{E:Zar-closed}, which concludes the proof.
\end{proof}

We can now define the non-archimedean topology on the set $\calX(R):=\Hom(\Spec R, $ $ \calX)$ (resp. $\calX(K):=\Hom(\Spec K,  \calX)$) of $R$-valued (resp. $K$-points)
points of any stack $\calX$ of finite type over $\Spec R$.

\begin{defi}[Non-archimedean topology]\label{D:nonarchtop}
\noindent
\begin{enumerate}[(i)]
\item \label{D:nonarchtop1} Let $X\to \Spec R$ be an affine scheme of finite type over $\Spec R$ and let $j:X\hookrightarrow \bbA^N_R$ be a closed embedding into the $N$-dimensional
affine space over $R$ for some $N$.
 The non-archimedean topology on the set of $R$-valued points of $X(R)$ is the subspace topology with respect to the natural inclusion
$j_R: X(R)\subseteq \bbA_R^N(R)=R^N$ and the non-archimedean topology on $R^N$.

In a similar way, we define the non-archimedean topology on $X(K)$.
\item \label{D:nonarchtop2} Let $\calX\to \Spec R$ be a (Artin) stack  of finite type over $\Spec R$ and choose an atlas  $f:X\to \calX$ (i.e. $f$ a surjective and
smooth morphism and $X$ is a scheme over $\Spec R$) of $\calX$ with $X$  affine and of finite type over $\Spec R$. The non-archimedean topology on the set $\calX(R)$ of $R$-valued
points of $\calX$ is the quotient topology with respect to the natural surjective map $f_R: X(R)\to \calX(R)$ and the non-archimedean topology on $X(R)$.

In a similar way, we define the non-archimedean topology on $\calX(K)$.
\end{enumerate}
\end{defi}
We leave to the reader the straightforward verification that the above definitions do not depend on the choices made, i.e. the embedding $j$ in \eqref{D:nonarchtop1}
and the atlas $f$ in \eqref{D:nonarchtop2}.

We can now prove the main result of this section. The result is certainly well-known to the experts (see \cite[\S 3.4.1]{Tem} for the case of strictly $K$-affinoid spaces and \cite[\S 5.2.4]{Tem} for the case of formal 
schemes over $\Spec R$) but we include a proof for the lack of a suitable reference in the case of stacks of finite type over $\Spec R$.

\begin{thm}\label{T:anticon-red}
Let $\calX$ a stack of finite type over $\Spec R$ and consider the reduction map
$$\red_{\calX}: \calX(R):=\Hom(\Spec R,\calX)\to \calX(k):=\Hom(\Spec k, \calX)$$
induced by composing with the map $s:\Spec k\to \Spec R$.
If we put the non-archimedean topology on $\calX(R)$ and the Zariski topology on $\calX(k)$, then $\red_{\calX}$ becomes  an anticontinuous map.
\end{thm}
\begin{proof}
We first prove the theorem in two special cases.

\un{Case I}: assume that $\calX=\bbA^n$ for some $n$.

In this case, the theorem reduces to Lemma \ref{L:redAn}.

\un{Case II}:Ê assume that $\calX$ is an affine scheme of finite type over $\Spec R$.

Choose a closed embedding $j:\calX\hookrightarrow \bbA^N$ for some $N$ as in Definition \ref{D:nonarchtop}\eqref{D:nonarchtop1}.
This induces a commutative diagram
\begin{equation}\label{E:diag-affine}
\xymatrix{
\calX(R)\ar[r]^{\red_{\calX}} \ar@{_{(}->}[d]_{j_R}& \calX(k)\ar@{^{(}->}[d]^{j_k}\\
\bbA^N(R) \ar[r]^{\red_{\bbA^N}}& \bbA^N(k), \\
}
\end{equation}
where the vertical arrows are injective.
If we put the non-archimedean topology on the sets on the left of the diagram and the Zariski topology on the sets on the right of the diagram, then we have that:
\begin{itemize}
\item $\red_{\bbA^N}$ is anticontinuous by Case I;
\item $j_R$ is continuous by Definition \ref{D:nonarchtop}\eqref{D:nonarchtop1};
\item $j_k$ is a closed continuous map since $j$ is an embedding.
\end{itemize}
Now, using the above facts, an easy diagram chase in \eqref{E:diag-affine}  shows that $\red_{\calX}$ is anticontinuous and Case II is proved.

Let us now consider an arbitrary stack $\calX$  of finite type over $\Spec R$.
Choose an atlas $f:X\twoheadrightarrow \calX$ with $X$ affine and of finite type over $\Spec R$ as in Definition \ref{D:nonarchtop}\eqref{D:nonarchtop2}.
This induces a commutative diagram
\begin{equation}\label{E:diag-atlas}
\xymatrix{
X(R) \ar[r]^{\red_{X}} \ar@{->>}[d]^{f_R}& X(k) \ar@{->>}[d]^{f_k}, \\
\calX(R)\ar[r]^{\red_{\calX}} & \calX(k)\\
}
\end{equation}
where the vertical arrows are surjective.
If we put the non-archimedean topology on the sets on the left of the diagram and the Zariski topology on the sets on the right of the diagram, then we have that:
\begin{itemize}
\item $\red_{X}$ is anticontinuous by Case II;
\item $f_R$ is a quotient map by Definition \ref{D:nonarchtop}\eqref{D:nonarchtop2};
\item $f_k$ is a  continuous map because it is induced by the morphism of stacks $f$.
\end{itemize}
Now, using the above facts,  an easy diagram chase in \eqref{E:diag-atlas} shows that $\red_{\calX}$ is anticontinuous, q.e.d.

\end{proof}

In the case of proper stacks $\calX$ over $\Spec R$, we can extend the reduction map to the set of $K$-valued points.

\begin{cor}\label{C:red-proper}
Let $\calX$ be a proper stack over $\Spec R$. Then the reduction map $\red_{\calX}$ can be extended to a map $\wt{\red}_{\calX}:\calX(K)\to \calX(k)$ as in the following diagram
$$\xymatrix{
\calX(R)\ar[r]^{\red_{\calX}} \ar@{^{(}->}[d]& \calX(k)\\
\calX(K)\ar[ru]_{\wt{\red}_{\calX}}
}
$$
Moreover, $\wt{\red}_{\calX}$ is anticontinuous with respect to  the non-archimedean topology on $\calX(K)$ and the Zariski topology on $\calX(k)$.
\end{cor}
\begin{proof}
Let us denote by $\calE$ the set of all the finite degree extensions $K\subseteq L$. For each $L\in \calE$, we denote by $R_L$ the valuation ring of $L$ with respect to the
unique extension of the valuation $\val$ on $K$ to a valuation $\val_L$ on $L$ (see \ref{S:notations}).

Since $K$ is complete with respect to the valuation $\val$,  for each $L\in \calE$ there is a
unique extension of the valuation $\val$ on $K$ to a valuation $\val_L$ on $L$, which is moreover still complete. We denote by $R_L\subset L$ the associated valuation ring.
Since $k$ is algebraically closed, the residue field of each of the rings $R_L$ (for $L\in \calE$) is equal to $k$. Therefore, we get a diagram
\begin{equation}\label{E:union-diag}
\bigcup_{L\in \calE} \calX(L)\stackrel{\eta}{\longleftarrow} \bigcup_{L\in \calE}\calX(R_L) \stackrel{\ov{\red}_{\calX}}{\longrightarrow }\calX(k).
\end{equation}
We endow the sets appearing in \eqref{E:union-diag} with the following topologies:  on $\calX(k)$ we put the Zariski topology; on $\bigcup_{L\in \calE} \calX(L)$
we put the finest topology for which all the inclusions $\calX(L) \hookrightarrow  \bigcup_{L\in \calE} \calX(L)$ are continuous with respect to the non-archimedean topology on $\calX(L)$;
the topology on  $\bigcup_{L\in \calE}\calX(R_L)$ is defined in a similar way. With respect to these topologies, the map $\eta$ is clearly continuous while the map $\ov{\red}_{\calX}$ is
anticontinuous by Theorem \ref{T:anticon-red}. Moreover the valuative criterion for properness of stacks applied to $\calX$ implies that $\eta$ is an homeomorphism.
We define a map $\wt{\red}_{\calX}:\calX(K)\to \calX(k)$  by composing the injection $\calX(K)\hookrightarrow \bigcup_{L\in \calE} \calX(L)$,
the homeomorphism $\eta^{-1}$  and the reduction map $\ov{\red}_{\calX}$. It is now clear the map $\wt{\red}_{\calX}$ satisfies all the required properties.

\end{proof}

\begin{proof}[Proof of Theorem \ref{T:mainthm2}]
It is easily checked that the reduction map $\red:\MMg(K)\to \MMgb(k)$ constructed in Lemma-Definition \ref{D:red-curves} is the restriction of the reduction
map $\wt{\red}_{\MMgb}:\MMgb(K)\to \MMgb(k)$ constructed in Corollary \ref{C:red-proper} to the open subset $\MMg(K)\subset \MMgb(K)$.  Therefore
the anticontinuity of $\red:\MMg(K)\to \MMgb(k)$ follows the anticontinuity of the reduction map  $\wt{\red}_{\MMgb}$ proved in Corollary \ref{C:red-proper}.

A similar argument applies to the reduction map $\red:\AAg(K)\to \AAgb(k)$ using the anticontinuity of the reduction map $\wt{\red}_{\AAgb}:\AAgb(K)\to \AAgb(k)$
(again from Corollary \ref{C:red-proper}).

\end{proof}

\section*{acknowledgments}

During the preparation of this manuscript and previous related papers, the author benefited from useful conversations with V. Alexeev, M. Baker, S. Brannetti, L. Caporaso, M. Chan, E. Cotterill, W. Gubler, M. Franciosi, M. Melo, J. Neves, S. Payne, J. Rabinoff. Moreover, the author wishes to thank the referee for useful comments and  for suggesting the references \cite{CR}, \cite{Ger}, \cite{Tem} and 
\cite{vdP}.

\vspace{0,1cm}

The author is also a member of the research center CMUC (University of Coimbra)
and he was supported by  the FCT project \textit{Espa\c cos de Moduli em Geometria Alg\'ebrica} (PTDC/MAT/111332/2009) and 
by the MIUR project  \textit{Spazi di moduli e applicazioni} (FIRB 2012).

\end{document}